\documentclass{amsart}

\usepackage{a4wide,amsmath,amssymb}
\usepackage[toc,page]{appendix}

\usepackage{url}
\usepackage{xcolor}

\newtheorem{thm}{Theorem}[section]
\newtheorem{lemma}[thm]{Lemma}

\newtheorem{cor}[thm]{Corollary}
\theoremstyle{remark}
\newtheorem{rem}[thm]{Remark}
\newtheorem{opp}[thm]{Open problem}

\newtheorem{conj}[thm]{Conjecture}
\theoremstyle{definition}

\newtheoremstyle{Claim}{}{}{\itshape}{}{\itshape\bfseries}{:}{ }{#1}
\theoremstyle{Claim}

\newcommand{\R}{\mathbb{R}}

\newcommand{\X}{\mathcal{X}}

\newcommand{\eps}{\varepsilon}

\theoremstyle{plain}

\def\sideremark#1{\ifvmode\leavevmode\fi\vadjust{
\vbox to0pt{\hbox to 0pt{\hskip\hsize\hskip1em
\vbox{\hsize3cm\tiny\raggedright\pretolerance10000
\noindent #1\hfill}\hss}\vbox to8pt{\vfil}\vss}}}

\begin{document}

\title[]{On the Liouville property for fully nonlinear equations with superlinear first-order terms}


\author{Marco Cirant}
\address{Dipartimento di Matematica ``Tullio Levi-Civita'', Universit\`a degli Studi di Padova, 
via Trieste 63, 35121 Padova (Italy)}
\curraddr{}
\email{cirant@math.unipd.it}
\thanks{}

\author{Alessandro Goffi}
\address{Dipartimento di Matematica ``Tullio Levi-Civita'', Universit\`a degli Studi di Padova, 
via Trieste 63, 35121 Padova (Italy)}
\curraddr{}
\email{alessandro.goffi@unipd.it}
\thanks{}

 \thanks{
 The authors are members of the Gruppo Nazionale per l'Analisi Matematica, la Probabilit\`a e le loro Applicazioni (GNAMPA) of the Istituto Nazionale di Alta Matematica (INdAM). The authors warmly thank the anonymous referees for their careful reading and the suggestions that improved the presentation of the first draft of the manuscript. The 
 authors wish to thank Prof. Roberta Filippucci and Dr. Giulio Tralli for fruitful discussions and for suggesting many references on the matter of this paper.
 }

\date{\today}

\begin{abstract}
We consider in this note one-side Liouville properties for viscosity solutions of various fully nonlinear uniformly elliptic inequalities, whose prototype is $F(x,D^2u)\geq H_i(x,u,Du)$ in $\R^N$, where $H_i$ has superlinear growth in the gradient variable. After a brief survey on the existing literature, we discuss the validity or the failure of the Liouville property in the model cases $H_1(u,Du)=u^q+|Du|^\gamma$, $H_2(u,Du)=u^q|Du|^\gamma$ and $H_3(x,u,Du)=\pm u^q|Du|^\gamma-b(x)\cdot Du$, where $q\geq0$, $\gamma>1$ and $b$ is a suitable velocity field. Several counterexamples and open problems are thoroughly discussed. 
\end{abstract}

\maketitle


\section{Introduction}
The question of non-existence of non-trivial solutions to linear and nonlinear partial differential equations and inequalities satisfying one or two-side bounds - usually called \textit{Liouville-type property} - has been widely investigated in the literature. This qualitative property is crucial in many theoretical aspects of PDEs, ranging from regularity theory to quantitative properties, see e.g. the monographs \cite{QS,VeronBook}.

The Liouville property for solutions to PDEs is usually obtained as a consequence of the (invariant) Harnack inequality, see e.g. \cite{CC, GT,KL,Landis}. Other ways to deduce such Liouville-type results are based on mean-value formulas \cite{BLU, Nelson} or on a priori gradient estimates, that can be derived either via maximum principle methods, 
as in \cite{PP}, or via Bernstein-type methods, as started in \cite{SP, Serrin}
, both for elliptic and parabolic equations, see also the recent work \cite{Chang} along with the discussion and references in \cite{CM}. We refer also to \cite{LiNew} for a quite different approach based on a comparison principle on punctured domains and to \cite{PZ} for probabilistic methods as well as for a discussion on control theoretic interpretations.

The case of solutions to partial differential inequalities satisfying a one-side bound is somehow different. It is well-known that superharmonic (subharmonic) functions bounded from below (above) in $\R^N$ are constants provided that $N\leq2$.  This result can be obtained via the Hadamard three-circle theorem \cite[Theorem 29 Ch. 2]{PW} or via different methods involving the existence of a so-called Lyapunov function  (alternatively known as Khas'minskii test, cf \cite{MPcomm} and the references therein), see \cite[Theorem 2.1 and Remark 2.2]{BG_lio1}, \cite[Theorem 2.1]{BC} and also \cite[p. 230-(ii)]{Pigola}. The underlying motive is the unboundedness of the fundamental solution of the Laplace equation.
This result is sharp as there exist counterexamples in higher dimensions, see the introduction in \cite{BG_lio1}. Nonetheless, it is also by now well-known that the perturbation of second-order operators with suitable nonlinearities, such as zero-th and first-order terms, allows to recover the Liouville property in higher dimensions.

The first analysis on supersolutions to semi-linear equations started with B. Gidas \cite{Gi}. When the equation is driven by a linear (or even quasi-linear) operator, the typical strategies to derive Liouville properties for inequalities with zero-th and/or first order terms are two. The first one was developed during the 90's by M.-F. Bidaut-V\'eron \cite{BVarma}, H. Berestycki- I. Capuzzo-Dolcetta- L. Nirenberg \cite{BCDN}, E. Mitidieri-S. Pohozaev \cite{MP,MPbook,DamMit} and V. Kurta \cite{KurtaProc,Kurta}, see also the references therein, and it is based on integral estimates and test function methods via the weak formulation of the problem. See also the references given in \cite{BMPR} for recent developments. 
 Moreover, such integral methods provide finer results that are known in the literature as Liouville comparison principles, as studied e.g. by V. Kurta and V. Kurta-B. Kawohl, see for instance \cite{KurtaProc,Kurta,KurtaKawohl}. When the Liouville property fails, the latter analysis allows even to establish the sharp distance at infinity of the non-constant supersolution bounded below by a constant to the constant itself, cf e.g. \cite[Theorem 1.2]{KurtaProc} and the references therein.
The second approach basically consists in ``radializing" the equation: when it is driven by the Laplacian, one can observe that the spherical mean of the unknown function satisfies again a partial differential inequality, reducing the analysis to an ODE problem, cf \cite[Section 7]{CaristiMitidieri}. We refer to the introduction of \cite{BG_lio1} and to \cite{BMPR,GrygCPAM} for further updated references on the quasi-linear and semi-linear (even degenerate) case for problems posed on Riemannian manifolds and Carnot groups.

However, the aforementioned methods break down in the fully nonlinear setting, and actually fail even for linear problems driven by operators in non-divergence form with non-differentiable diffusion matrix and/or having unbounded drifts, such as those of Ornstein-Uhlenbeck type.  To the authors' knowledge, there are few established approaches to handle one-side Liouville properties for fully nonlinear partial differential inequalities, and all of them are obviously based on maximum principle arguments, as the right setting to investigate such equations is that of viscosity solutions \cite{CC,CIL}. 

The first one was initiated by A. Cutr\`i and F. Leoni \cite{CLeoni}. They proved non-existence of non-trivial solutions (in the viscosity sense) to $F(x,D^2u)\geq u^q$ in $\R^N$ (or more generally to problems driven by rotationally invariant operators perturbed with a weighted radial zero-th order term) when $q$ is below a certain exponent related to the scaling of the equation, via a nonlinear Hadamard three-sphere theorem. This approach is based on the following observation: any nonnegative solution to $F(x,D^2u)\geq u^q$ in $\R^N$ is also a solution to $\mathcal{P}^+_{\lambda,\Lambda}(D^2u)\geq0$ in $\R^N$, where $\mathcal{P}^+_{\lambda,\Lambda}$ stands for the Pucci's maximal operator (see below for the definition) and it satisfies by the maximum principle $\min_{\partial B(0,r)} u=\min_{\bar{B}(0,r)}u$, which implies that the function $r\longmapsto m(r)=\min_{\partial B(0,r)}u$ is decreasing. Then, owing to the radial symmetry of the fundamental solution of the extremal operators, one can establish a nonlinear Hadamard three-sphere theorem, and then deduce the crucial fact that the function $r^{\beta-2}m(r)$ is increasing, $\beta-2$ being the scaling exponent of the Pucci's maximal operator (which will be specified below), see \cite[Corollary 3.1]{CLeoni}. Finally, the interplay between the comparison principle, the strong minimum principle and the construction of a suitable test function allows to prove that $m(r)$ approaches zero as $r \to \infty$ with a certain rate, contradicting the previous step. 
 Alternatively, the monotonicity of the function $r^{\beta-2}m(r)$ can be combined with the weak Harnack inequality from \cite{CC} to get the same conclusion, cf e.g. \cite[Remark 9]{CLeoni}. The approach via the Hadamard three-sphere theorem was later refined in \cite{CDC2003} for problems with gradient terms having at most linear growth, under appropriate smallness conditions on the drift coefficient. See also \cite{FQS} for Liouville theorems under Keller-Osserman type conditions. Other related results based on qualitative properties of the function $m(r)$ have been obtained in \cite{BirDem,BirDemEJDE,QuaasJMPA,QuaasAnnali}, see also \cite{BGL,Goffi} for further refinements for degenerate problems.

Other results for fully nonlinear equations having at most linear growth in the gradient have been obtained in \cite{Rossi}, under suitable conditions at infinity on the coefficients. The approach is based on a strong maximum principle and the construction of suitable test functions. We refer to the references given in \cite{Rossi} for further results on semi-linear and quasi-linear equations under similar conditions on the coefficients. 

More general approaches to handle Bellman-Isaacs equations (even those driven by the $p$-Laplacian and/or set on exterior domains) have been developed in \cite{AScpde, ASannali} by S. Armstrong and B. Sirakov. The results obtained in \cite{ASannali} are based again on maximum principles arguments and the analysis of ``fundamental solutions" to such operators initiated in \cite{ArmstrongCPAM}.  These somehow extended \cite[Theorem 4.1]{CLeoni} to a wider framework when the equation is perturbed by zero-th order terms. Instead, the results in \cite{AScpde} are based on a quantitative version of the strong maximum principle, cf \cite[Theorem 3.3]{AScpde}, and the method of proof allows to encompass both equations driven by Bellman-Isaacs and $p$-Laplacian operators. We emphasize that the dividing exponent implying the Liouville property in the general case of a Bellman-Isaacs operator is not explicit as in the case of rotationally invariant operators, as it depends on the intrinsic dimension of its fundamental solution.

It is well-known that the Liouville property for linear operators is strictly related to the recurrence of the associated diffusion process, which is in turn tied up with the existence of a so-called \textit{Lyapunov function}, see for example \cite{Gry1,MV} and references therein. The argument showing that the Liouville property follows from the existence of Lyapunov functions (or the validity of the {Khas'minskii test) is particularly flexible: it basically needs the strong maximum/minimum principle and the comparison principle on bounded open sets, and does not go through more sophisticated steps such as  regularity properties or Harnack inequalities (as \cite{AScpde,ArmstrongCPAM}). Maximum and comparison principles are two ingredients that are broadly available for elliptic PDEs (cf \cite{BG,GP} for recent developments in the viscosity solutions' framework). This strategy has been consequently extended to the nonlinear framework in the context of  Riemannian manifolds on one hand, see \cite{BMPR,MPcomm,PigolaProc,PigolaRMI,Pigola}, and, on the other hand, in \cite{BC} and later in \cite{BG_lio1,BG_lio2}, with focus on degenerate problems, possibly modeled over H\"ormander vector fields. We refer also to \cite{LionsMusiela} and \cite{BC,MMT,KLP} for related Liouville properties and applications to ergodic control.

We remark in passing that the non-existence properties in the fully nonlinear setting transfer immediately to non-divergence equations (i.e. involving the operator $-\mathrm{Tr}(A(x)D^2u)$), being the Pucci's extremal operators the prototype ones in non-divergence form, see Remark \ref{nondiv}. We emphasize that the literature on Liouville properties driven by non-divergence operators is poor to our knowledge, see e.g. \cite{Sobol} for some results in this direction for semi-linear equations.

Finally, we mention that different Liouville properties for partial differential equations and inequalities have been proved for solutions belonging to $L^p$ scales: for this topic, without being exhaustive, we refer to \cite[Section 13.1]{Gry1}, \cite{KLna}, \cite{Pigola} and the references therein. 

\medskip

In this note, our main motivations in Liouville-type results are their implications in regularity properties and existence of solutions to nonlinear elliptic equations \cite{CGell, GS,Ruiz} (that can be established via blow-up procedures), their qualitative behavior \cite{Bianchi, VeronBook}, ergodic control problems and the large time behavior for nonlinear parabolic equations, cf e.g. \cite{BC}, and the recent developments in the theory of Mean Field Games \cite{CirantCPDE}, where systems of equations with gradient dependent nonlinearities appear naturally. In particular, we aim at investigating Liouville properties for supersolutions bounded below of fully nonlinear equations, with particular emphasis on problems involving a first-order term with superlinear growth in the gradient. 
We will consider the following inequality in the viscosity sense,
 \[
 F(x,D^2u)\geq H_i(x,u,Du)\text{ in }\R^N,
 \]
 where $F$ is uniformly elliptic, i.e.
 \[
 \mathcal{P}^-_{\lambda,\Lambda}(M-Q)\leq F(x,M)-F(x,Q)\leq \mathcal{P}^+_{\lambda,\Lambda}(M-Q)\ ,M,Q\in\mathrm{Sym}_N\ ,Q\geq0\ ,
 \]
 where $\mathrm{Sym}_N$ stands for the space of $N\times N$ symmetric matrices.
Here, $\mathcal{P}^\pm_{\lambda,\Lambda}$ are the Pucci's extremal operators \cite{CC} defined for ellipticity constants $0<\lambda\leq \Lambda$ as
 \[
 \mathcal{P}^+_{\lambda,\Lambda}(M)=\sup_{\lambda I_N\leq A\leq \Lambda I_N}-\mathrm{Tr}(AM)=-\lambda\sum_{e_k>0}e_k-\Lambda\sum_{e_k<0}e_k\ ,
 \] 
 \[
 \mathcal{P}^-_{\lambda,\Lambda}(M)=\inf_{\lambda I_N\leq A\leq \Lambda I_N}-\mathrm{Tr}(AM)=-\Lambda\sum_{e_k>0}e_k-\lambda\sum_{e_k<0}e_k\ ,
 \]
 where $e_k$ stands for the $k$-th eigenvalue of the matrix $M$. We will further assume the normalization condition $F(x,0)=0$, and discuss the validity of the Liouville property for the following model nonlinearities
\begin{align*}
H_1(r,\xi)&=r^q+|\xi|^\gamma\ , \\ H_2(r,\xi)&=r^q|\xi|^\gamma\ ,\\ H_3(x,r,\xi)&=\pm r^q|\xi|^\gamma-b(x)\cdot \xi\ ,
\end{align*}
with $q\geq0$, $\gamma>1$, and $b$ being a suitable confining velocity field. In the first two cases it clearly holds that $\mathcal{P}^+_{\lambda,\Lambda}(D^2u)\geq0$, and this fact will be frequently exploited.

We now recall some related results for problems driven by linear operators. First, it is worth recalling that when $H=H_2$ and $\lambda=\Lambda=1$, in the case $q=0$ the equation reduces to the so-called \textit{viscous Hamilton-Jacobi equation} \cite{Lions85,CGell}, while when $\gamma=0$ it reduces to the celebrated \textit{Lane-Emden equation} \cite{Gi} (see \cite[Theorem 1.1]{ASannali} for results in this direction). For $q\geq0$, $\gamma>1$, it has been first proved in \cite{CaristiMitidieri} that classical solutions to
\[
-\Delta u\geq u^q|Du|^\gamma\text{ in }\R^N\ ,N>2\ ,
\]
must be constants provided that
\[
(N-2)q+(N-1)\gamma\leq N\ .
\]
The proof in \cite[Theorem 7.1]{CaristiMitidieri} is based on replacing $u$ by its spherical mean, which turns out to be a supersolution by means of Jensen's inequality. This allows to reduce the study of the Liouville property to radial supersolutions, hence to an ODE analysis. Instead, a different study was performed in e.g. \cite{KurtaKawohl} and \cite[Theorem 15.1]{MP} via the nonlinear capacity method combined with an asymptotic analysis, see also \cite[Corollary 3]{MelianDCDS}, the recent papers \cite[Theorem 2.1]{VeronDuke}, \cite[Theorem 1.3]{Chang}, and the earlier results in \cite[Corollary 1]{FilippucciJDE2011} together with the further developments obtained therein for elliptic systems of inequalities with weights, even driven by $p$-Laplacians. We finally mention the recent paper \cite{Sun} which studies the elliptic inequality $-\Delta_p u\geq u^q|Du|^\gamma$ on Riemannian manifolds under volume growth conditions, even for negative values of $q,\gamma$. Note that these latter results for negative pairs are new even in the Euclidean case. We remark that \cite[Corollary 1.1]{FilippucciNA} provides the result when $\gamma<1$ and for operators driven by the $p$-Laplacian, see \cite{BirDem} and \cite{MelianDCDS} for a different proof based on the maximum principle.\\
In the fully nonlinear case, the only available Liouville property for the inequality $F\geq H_2$ in $\R^N$ was obtained in \cite{BirDem} when $q>0$ and $0<\gamma<1$, see Remark \ref{rembir}.\\
The case $H=H_1$ driven by the Laplacian in exterior domains has been first tackled in \cite[Theorem 7.2]{CaristiMitidieri} again via the radialization through the spherical mean. The same problem has been recently investigated in much more generality in \cite{QuaasJMPA} via the methods initiated in \cite{CLeoni}, where the Liouville property for
\[
-\Delta u\geq u^q+|Du|^\gamma \quad \text{ in }\R^N\backslash B(0,R_0)
\]
has been proved when $q>\frac{N}{N-2}$ and $\gamma\leq \frac{N}{N-1}$ (the case $q\leq \frac{N}{N-2}$ and $\gamma>0$ being a trivial consequence of the result for supersolutions to $-\Delta u\geq u^q$). This result is sharp, since one can find counterexamples to the Liouville property when $\gamma>\frac{N}{N-1}$, see Remark \ref{H1fail} for an explicit calculation.

The case $H=H_3$ with $q=\gamma=0$ and $F=\mathcal{P}^+$ has been first tackled in \cite{BC} in the Euclidean case and in \cite{BG_lio1} for subelliptic problems. In particular, assuming the existence of a sort of Lyapunov function, namely of a viscosity subsolution $w$ outside a compact set such that $\lim_{|x|\to\infty}w(x)=-\infty$, it has been proved that if $u\in\mathrm{LSC}$ is a supersolution to \[\mathcal{P}^+_{\lambda,\Lambda}(D^2u)+\sup_{\alpha\in A}\{c^\alpha(x)u-b^\alpha(x)\cdot Du\}=0\text{ in }\R^N\] satisfying $\limsup_{|x|\to\infty}\frac{u(x)}{w(x)}\leq 0$, then $u$ is constant. The existence of such $w$ is shown, for example, whenever
\[
\sup_{\alpha\in A}\{b^\alpha(x)\cdot x-c^\alpha(x)|x|^2\log|x|\}\leq \lambda-\Lambda(N-1)
\]
for large $|x|$.
\medskip

In this note we will discuss the following results, which depend on the \textit{effective dimension} \[\beta=\frac{\Lambda}{\lambda}(N-1)+1\] of the Pucci's maximal operator, which is in turn related to the scaling exponents of the nonlinearities.
\begin{itemize}
\item When $H=H_1$ and $F$ is uniformly elliptic, we prove the validity of the Liouville property when $N\geq 3$, and either $q\leq \frac{\beta}{\beta-2}$ and $\gamma>1$ (actually $\gamma>0$ is allowed), or $q>\frac{\beta}{\beta-2}$ and $1<\gamma\leq \frac{\beta}{\beta-1}$, see Theorem \ref{mainsum}. We further prove the sharpness of the result when the above growth conditions fail, see Remark \ref{H1fail}. To our knowledge, these properties are new, although the proof combines known ideas from \cite{CLeoni,QuaasJMPA}. Indeed, nonnegative solutions to $F\geq H_1$ are also solutions to $\mathcal{P}^+_{\lambda,\Lambda}(D^2u)\geq u^q$ in $\R^N$, a fact that is crucial to establish the Liouville property;
\item When $H=H_2$, we prove that the Liouville property fails when 
\[
(\beta-2)q+(\beta-1)\gamma>\beta\ ,q\geq0\ ,\gamma>1.
\]
In particular, when $q=0$ we prove in Theorem \ref{countexq} that for a suitable constant $c>0$ the function $u(x)=c(1+|x|^2)^{-\frac{2-\gamma}{2(\gamma-1)}}$ is a classical solution to $\mathcal{P}^+_{\lambda,\Lambda}(D^2u)\geq |Du|^\gamma$ in $\R^N$, provided that
\[
\gamma>\frac{\beta}{\beta-1}\ .
\]
Based on this observation, we conjecture its validity when
\[
(\beta-2)q+(\beta-1)\gamma\leq\beta, q\geq0\ ,\gamma>1,
\]
and in particular that for the viscosity inequality $F(x,D^2u)\geq |Du|^\gamma$ in $\R^N$ the Liouville property holds provided that $1<\gamma\leq\frac{\beta}{\beta-1}$. This problem will be addressed in a future research.
\item When $H=H_3$ and $q=0$, we provide a sufficient condition to the Liouville property involving the behavior of the drift $b$ 
, for problems with arbitrary ``repulsive'' gradient terms, or quadratic ``absorbing" terms (via a nonlinear Hopf-Cole transformation), see Theorem \ref{rep}. These properties are new, and will be derived from the results in \cite{BC,BG_lio1}, which use Lyapunov functions (Khas'minskii test). We will show in Theorem \ref{rep2} that such condition can be improved in some special cases (such as for supersolutions to equations driven by $F= \mathcal{P}^-$). Moreover, a variant of the Hopf-Cole transformation for fully nonlinear equations allows to treat the case $H(u,Du)=u^q|Du|^2$, see Remark \ref{mixquad}. For these drift perturbed equations, we show that the condition ensuring the Liouville property is sharp, see Remark \ref{sharpdrift}. Therefore, we prove that our nonlinear version of the Khas'minskii test is not only sufficient, but also necessary for the validity of the Liouville property, at least in the two cases $F= \mathcal{P}^\pm$.
\item In Section \ref{sec:rem} we make some remarks on degenerate equations, on the cases $N=1,2$ for $H_i$, $i=1,2,3$, and on some problems driven by non-divergent operators. Finally, we make in Section \ref{sec:sol} some remarks on the Liouville property for solutions to Hamilton-Jacobi equations. Some open problems will also be mentioned throughout the paper.
\end{itemize}
Finally, let us stress that some of our results are new even for supersolutions to the equation
\[
 -\mathrm{Tr}(A(x)D^2u)=H_i(x,u,Du)\text{ in }\R^N\ ,
\]
with $\lambda I_N\leq A\leq \Lambda I_N$, cf Remark \ref{nondiv}.
We defer to a future study the proof of the Liouville property for entire solutions to
\[
F(x,D^2u)\geq u^q|Du|^\gamma\text{ in }\R^N
\]
as well as for a general treatment of the aforementioned properties for Bellman-Isaacs equations. 
\section{Liouville properties for fully nonlinear equations involving the sum of first- and zero-th order terms with power-like growth}\label{sec;sum}
In this section we consider Liouville properties for solutions to viscosity inequalities in the whole space such as
\begin{equation}\label{Fgrad+0}
F(x,D^2u)\geq u^q+ |Du|^\gamma \quad \text{ in }\R^N\ ,
\end{equation}
where $F:\R^N\times\mathrm{Sym}_N\to\R$ is continuous, uniformly elliptic and satisfies the normalization condition $F(x,0)=0$, while $q>0$, $\gamma>1$. 
First, by the definition of uniform ellipticity, one observes that viscosity solutions to \eqref{Fgrad+0} are also solutions to the viscosity inequality
\[
\mathcal{P}^+_{\lambda,\Lambda}(D^2u)\geq u^q\text{ in }\R^N\ .
\]
Then, \cite[Theorem 1.4-(i)]{ASannali} or \cite[Theorem 4.1]{CLeoni} ensure that \eqref{Fgrad+0} has no nontrivial nonnegative supersolutions in $\R^N$ when $0<q\leq \frac{\beta}{\beta-2}$, where \[\beta=\frac{\Lambda}{\lambda}(N-1)+1\ .\] Such a threshold is sharp, since for $q>\frac{\beta}{\beta-2}$ the function $u(x)=c(1+|x|^2)^{-\frac{1}{q-1}}$ satisfies in  the above inequality for a suitable $c>0$ . 
Therefore, we restrict our attention to the case $q>\frac{\beta}{\beta-2}$, and add the nonlinearity $|Du|^\gamma$ to the problem. We characterize below the (sharp) interval for $\gamma$ that guarantees the Liouville property. 

The next result shows indeed that the presence of a (slightly) superlinear first order term allows to extend the range of the validity of the Liouville property even beyond the threshold $q=\frac{\beta}{\beta-2}$.

\begin{thm}\label{mainsum}
Assume that $q>\frac{\beta}{\beta-2}$ and $1<\gamma\leq \frac{\beta}{\beta-1}$.
Then, any nonnegative viscosity supersolution to \eqref{Fgrad+0} is constant.
\end{thm}


The proof of Theorem \ref{mainsum} will be based on the analysis of $$m(R):=\min_{|x|\leq R}u(x)$$, where $u$ is the supersolution to \eqref{Fgrad+0}.
for $R \ge 0$, which combines the strong minimum principle, see e.g. \cite{BG,GP}, and the comparison principle \cite{CIL}. We will follow some ideas developed in \cite{CLeoni} and \cite{QuaasJMPA}. Note that by the strong minimum principle we can always assume $u>0$ in $\R^N$, so that $m(R)>0$. Indeed, if there existed a point $x_0\in B(0,R)$ such that $u$ attains its minimum at $x_0$, it would follow $u\equiv0$. Then, by the arbitrariness of $R>0$ we would have $u\equiv0$ in $\R^N$.

We start with the following algebraic result.
\begin{lemma}\label{prepP}
Let $u^\delta(x)=C_\delta(1+|x|^2)^{-\frac{\delta}{2}}$ for $\delta,C_\delta>0$. Then
\[
\mathcal{P}^+_{\lambda,\Lambda}(D^2u^\delta)\geq C_\delta\delta\lambda \frac{\beta-\delta-2}{(1+|x|^2)^{\frac{\delta}{2}+1}}\ .
\]
Moreover, for $v^\nu(x)=\Theta|x|^{-\nu}$, $\Theta,\nu>0$, we have for $x\in\R^N\backslash\{0\}$
\[
\mathcal{P}^+_{\lambda,\Lambda}(D^2v^\nu)=-\lambda\nu \Theta(\nu+1)|x|^{-\nu-2}+\nu \Theta\Lambda|x|^{-\nu-2}(N-1)\ .
\]
\end{lemma}
\begin{proof}
First, recall that if $f=f(|x|)$, the eigenvalues of $D^2f$ are $f''(|x|)$, which is simple, $\frac{f'(|x|)}{|x|}$ with multiplicity $N-1$. Then, denoting by $f^\delta(|x|)=C_\delta(1+|x|^2)^{-\frac{\delta}{2}}$, we have
\[
(f^\delta)'(|x|)=-C_\delta\delta(1+|x|^2)^{-\frac{\delta}{2}-1}|x|\ ,
\]
\begin{align*}
(f^\delta)''(|x|)&=C_\delta\delta(\delta+2)|x|^2(1+|x|^2)^{-(\frac{\delta}{2}+2)}-C_\delta\delta(1+|x|^2)^{-(\frac{\delta}{2}+1)}\\
&=C_\delta\delta (1+|x|^2)^{-(\frac{\delta}{2}+2)}[(\delta+2)|x|^2-(1+|x|^2)]\\
&=C_\delta\delta (1+|x|^2)^{-(\frac{\delta}{2}+2)}[(\delta+1)|x|^2-1]\ .
\end{align*}
In view of the definition of the Pucci's maximal operator
\[
\mathcal{P}^+_{\lambda,\Lambda}(D^2u)=-\lambda\sum_{e_k>0}e_k-\Lambda\sum_{e_k<0}e_k
\]
we first address the case $|x|^2\leq\frac{1}{\delta+1}$, where $u^\delta(|x|)$ is concave and decreasing. Then, exploiting the fact that $\Lambda(N-1)=\lambda(\beta-1)$, we have
\begin{multline*}
\mathcal{P}^+_{\lambda,\Lambda}(D^2u^\delta)=-\Lambda C_\delta\delta (1+|x|^2)^{-(\frac{\delta}{2}+2)}[(\delta+1)|x|^2-1]+\Lambda C_\delta \delta(1+|x|^2)^{-(\frac{\delta}{2}+1)}(N-1)\\
=C_\delta \delta\Lambda\left[\frac{N-1}{(1+|x|^2)^{\frac{\delta}{2}+1}}\underbrace{-\frac{(\delta+1)|x|^2-1}{(1+|x|^2)^{\frac{\delta}{2}+2}}}_{\geq0}\right]
\geq C_\delta\delta\frac{\Lambda(N-1)}{(1+|x|^2)^{\frac{\delta}{2}+1}}=C_\delta\delta \frac{\lambda(\beta-1)}{(1+|x|^2)^{\frac{\delta}{2}+1}}\ .
\end{multline*}
When instead $|x|^2\geq \frac{1}{\delta+1}$ we write
\begin{align*}
\mathcal{P}^+_{\lambda,\Lambda}(D^2u^\delta)&=-\lambda C_\delta\delta (1+|x|^2)^{-(\frac{\delta}{2}+2)}[(\delta+1)|x|^2-1]+\Lambda C_\delta \delta(1+|x|^2)^{-(\frac{\delta}{2}+1)}(N-1)\\
&=C_\delta \delta\lambda\left[\frac{\beta-1}{(1+|x|^2)^{\frac{\delta}{2}+1}}-\frac{(\delta+1)|x|^2-1}{(1+|x|^2)^{\frac{\delta}{2}+2}}\right]\\
&=C_\delta \delta\lambda\left[\frac{(\beta-1)(1+|x|^2)-(\delta+1)|x|^2+1}{(1+|x|^2)^{\frac{\delta}{2}+2}}\right]\\
&= C_\delta \delta\lambda\left[\frac{(\beta-\delta-2)(1+|x|^2)+\beta}{(1+|x|^2)^{\frac{\delta}{2}+2}}\right]\\
&\geq C_\delta\delta\lambda \frac{\beta-\delta-2}{(1+|x|^2)^{\frac{\delta}{2}+1}}\ .
\end{align*}
where in the last inequality we used that $(\beta-\delta-2)|x|^2+\beta> (\beta-\delta-2)(1+|x|^2)$. In both cases, we get
\[
\mathcal{P}^+_{\lambda,\Lambda}(D^2u^\delta)\geq C_\delta\delta\lambda \frac{\beta-\delta-2}{(1+|x|^2)^{\frac{\delta}{2}+1}}\ .
\]
The second statement follows similarly, since $g^\nu(|x|)=\Theta|x|^{-\nu}$, $(g^\nu)'=-\Theta\nu|x|^{-\nu-1}<0$ and $(g^\nu)''=\nu(\nu+1)\Theta|x|^{-\nu-2}>0$.
\end{proof}

We now state a few lemmata concerning the behavior of $m(R)$.

\begin{lemma}\label{mRincr}  The map $R \mapsto m(R)$ has the following properties.
\begin{itemize}
\item[{\it i)}] If $\mathcal{P}^+_{\lambda,\Lambda}(D^2u)\geq 0$ on $\R^N$, then $m(R)R^{\beta-2}$ is nondecreasing. 
\item[{\it ii)}] If $\mathcal{P}^+_{\lambda,\Lambda}(D^2u) \geq  |Du|^\gamma$ on $\R^N$, $m(R) \to 0$ as $R \to \infty$ and $\gamma\leq \frac{\beta}{\beta-1}$, then $m(R)R^{\nu}$ is nondecreasing for any $\nu \in (0, \beta-2)$ and $R$ large enough. 
\end{itemize}

\end{lemma}

\begin{proof} The proof of the two items follow similar lines. We start from {\it i)}, and follow \cite[Theorem 4.1]{CLeoni}. Fix $0 < R_1 < R$, we claim that the function \[\Psi(x)=\frac{m(R_1)-m(R)}{R_1^{-(\beta-2)}-R^{-(\beta-2)}}(|x|^{-(\beta-2)}-R^{-(\beta-2)})+m(R)
\] satisfies in the classical sense the equality $\mathcal{P}^+_{\lambda,\Lambda}(D^2\Psi) = 0$ for $|x|>R_1$. Indeed, setting $\Theta = \frac{m(R_1)-m(R)}{R_1^{-(\beta-2)}-R^{-(\beta-2)}}$, by Lemma \ref{prepP} and the fact that $\Lambda(N-1)=\lambda(\beta-1)$ we obtain
\[
\mathcal{P}^+_{\lambda,\Lambda}(D^2\Psi)=-\lambda(\beta-2) \Theta(\beta-1)|x|^{-\beta}+(\beta-2) \Theta\Lambda|x|^{-\beta}(N-1) = 0\ .
\]
We then apply the comparison principle to the viscosity / classical inequalities $$\mathcal{P}^+_{\lambda,\Lambda}(D^2u)\geq u^q > \mathcal{P}^+_{\lambda,\Lambda}(D^2\Psi)$$ in the annulus $A_{R_1,R}=\{R_1<|x|<R\}$ (cf \cite[Proposition 5]{FQS}). Since $u\geq \Psi$ on $\partial A_{R_1,R}$ by construction we get $u\geq \Psi$ in $A_{R_1,R}$. Moreover, since $m(R) \ge 0$,
$$u(x)\geq \Psi(x) \geq \frac{m(R_1)\big(|x|^{-(\beta-2)}-R^{-(\beta-2)}\big)}{R_1^{-(\beta-2)}-R^{-(\beta-2)}},$$
and {\it i)} follows by letting $R \to \infty$.

\smallskip

The proof of {\it ii)} is a bit more delicate. As before, we need to show that 
\[\Psi(x)=\frac{m(R_1)-m(R)}{R_1^{-\nu}-R^{-\nu}}(|x|^{-\nu}-R^{-\nu})+m(R)
\]
is a classical subsolution of $\mathcal{P}^+_{\lambda,\Lambda}(D^2\Psi)  \leq |D\Psi|^\gamma$ in $A_{R_1,R}$. Setting $\Theta = \frac{m(R_1)-m(R)}{R_1^{-\nu}-R^{-\nu}}$ (which is positive, since $m(R)$ is decreasing), by the identity $|D\Psi|^\gamma=(\nu \Theta)^\gamma|x|^{-(\nu+1)\gamma}$ we have such an inequality provided that
\[
|x|^{-\nu-2}\nu \Theta[-\lambda(\nu+1)+\lambda(\beta-1)-(\nu \Theta)^{\gamma-1}|x|^{-(\nu+1)\gamma+\nu+2}]\leq 0\ ,
\]
namely
\[
\Theta |x|^{-\nu + \frac{2-\gamma}{\gamma-1} }\geq\frac{\lambda^{\frac{1}{\gamma-1}}(\beta-\nu-2)^{\frac{1}{\gamma-1}}}{\nu}.
\]
Under the condition $1<\gamma\le\frac{\beta}{\beta-1}$, it is immediate to verify that $\frac{2-\gamma}{\gamma-1}\ge\beta-2$, hence $-\nu + \frac{2-\gamma}{\gamma-1} > 0$. Recalling the definition of $\Theta$, then the previous inequality holds true whenever $|x| \ge R_1$ and
\begin{equation}\label{crucineq}
m(R_1)R_1^{\frac{2-\gamma}{\gamma-1}} \ge \big(1-(R_1/R)^{\nu}\big) \frac{\lambda^{\frac{1}{\gamma-1}}(\beta-\nu-2)^{\frac{1}{\gamma-1}}}{\nu} + m(R)R_1^{\frac{2-\gamma}{\gamma-1}} .
\end{equation}
On one hand, by {\it i)}
\[
m(R_1)R_1^{\frac{2-\gamma}{\gamma-1}} \ge m(R_1)R_1^{\beta - 2} \ge m(1)
\]
for all $R_1 \ge 1$. On the other hand, note first that we can assume $m(1) > 0$, otherwise by the strong maximum principle $u \equiv 0$ on $\R^N$ and the statement would be trivial. Since $m(R) \to 0$, $m(R)R_1^{\frac{2-\gamma}{\gamma-1}} \le \frac{m(1)}2 $ for $R$ large enough (depending on $R_1$ and $m(1)$) and thus
\begin{multline*}
\big(1-(R_1/R)^{\nu}\big) \frac{\lambda^{\frac{1}{\gamma-1}}(\beta-\nu-2)^{\frac{1}{\gamma-1}}}{\nu} + m(R)R_1^{\frac{2-\gamma}{\gamma-1}} \le\\ \frac{\lambda^{\frac{1}{\gamma-1}}(\beta-\nu-2)^{\frac{1}{\gamma-1}}}{\nu} + \frac{m(1)}2 \le m(1)
\end{multline*}
provided that $\nu \in [\bar \nu, \beta - 2)$, and $\bar \nu$ is sufficiently close to $\beta -2$ (depending on $\lambda, \gamma, m(1)$ only, and not on $R_1$). Therefore, \eqref{crucineq} is verified, and thus $\Psi$ is a subsolution as desired. Arguing by comparison as in {\it i)}, we obtain that $m(R)R^{\nu}$ is nondecreasing for $R \ge 1$ and such $\nu$.

To extend the statement to any $\nu \in (0, \beta - 2)$, note that
$$
m(R_1)R_1^{\frac{2-\gamma}{\gamma-1}} \geq m(R_1)R_1^{\bar \nu}R_1^{\beta - 2 - \bar \nu} \ge m(1) R_1^{\beta - 2 - \bar \nu} \to \infty \quad\text{as $R_1 \to \infty$,}
$$
and therefore \eqref{crucineq} is verified whenever $R_1$ is large enough (now depending on $\lambda, \gamma, m(1)$, $\beta, \nu$) and $R$ is large enough (depending on $R_1$). We proceed as before by comparison and let $R \to \infty$ to get the statement.

\end{proof}

\begin{lemma}\label{mbound}
Let $u$ be a nonnegative solution to the viscosity inequality $$\mathcal{P}^+_{\lambda,\Lambda}(D^2u) \geq u^q$$ in $\R^N$. Then, 
the following  bound from above holds
\begin{equation}\label{leo}
m(R)\leq CR^{-\frac{2}{q-1}}\ ,\ C>0\ .
\end{equation}
\end{lemma}

\begin{proof}

The result has been proved in \cite[Theorem 4.1]{CLeoni}. We sketch the proof here for the reader's convenience (in \cite{CLeoni} some additional weights in the problem appear, but we do not have them here). For fixed $R> r$, we consider the test function
\[
\varphi(x)=m(r)\left\{1-\frac{[(|x|-r)^+]^3}{(R-r)^3}\right\}\ .
\]
We observe that $\varphi\leq 0<u$ for $|x|\geq R$, while $\varphi(x)\equiv m(r)<u(x)$ for $|x|<r$. Since  $\varphi(x)=u(x)$ at least at one point verifying $|x|=r$, we conclude that $u-\varphi$ achieves the nonpositive minimum in $\R^N$ at a certain point $\bar x$ in $\{x\in\R^N:r\leq |x|<R\}$. 
Thus we can use $\varphi$ as a test function in the definition of viscosity solution to find
\[
\mathcal{P}^+_{\lambda,\Lambda}(D^2\varphi(\bar x))\geq u^q(\bar x)\ .
\]
By straightforward calculations as in Lemma \ref{prepP} we get
\[
\mathcal{P}^+_{\lambda,\Lambda}(D^2\varphi(\bar x))=\frac{3\Lambda m(r)}{(R-r)^3}\left[2+(N-1)\frac{(|x|-r)^+}{|x|}\right](|x|-r)^+
\]
so that
\[
\frac{3\Lambda m(r)}{(R-r)^3}\left[2+(N-1)\frac{(|\bar x|-r)^+}{|\bar x|}\right](|\bar x|-r)^+\geq u^q(\bar x)\ .
\]
If $|\bar x|=r$, then we would find $u(\bar x)=0$, which is a contradiction. So we may assume $r<|\bar x|<R$. In this case, we conclude
\[
u^q(\bar x)\leq \frac{3\Lambda(N+1)m(r)}{(R-r)^2}\ .
\]
We then exploit the fact that $u(\bar x)\geq \min_{r<|x|<R}u(x)\geq m(R)$ and get
\[
m^q(R)\leq \frac{3\Lambda(N+1)m(r)}{(R-r)^2}\text{ for }R> r\ .
\]
In view of Lemma \ref{mRincr} we use that $m(r)\leq \frac{m(R)R^{\beta-2}}{r^{\beta-2}}$ and conclude
\[
m^q(R)\leq \frac{3\Lambda(N+1)m(R)R^{\beta-2}}{r^{\beta-2}(R-r)^2}\text{ for }R> r\ .
\]
We then take $r = R/2$ to get
\[
m(R)\leq \frac{[3\Lambda(N+1)m(R)]^\frac1q2^{\frac{1}{q}+\frac{\beta-2}{q}}}{R^{\frac{2}{q}}}\ ,
\]
that yields \eqref{leo}.

\end{proof}

We are now ready to prove the main theorem of this section.

\begin{proof}[Proof of Theorem \ref{mainsum}]
Assume by contradiction that $u$ is not constant, so $u > 0$ on $\R^N$ by the strong minimum principle, and $m(R_1) > 0$ for all $R_1$. Since $u$ is a supersolution to \eqref{Fgrad+0}, then it is also a supersolution to the equation $\mathcal{P}^+_{\lambda,\Lambda}(D^2u)\geq u^q$ on the whole $\R^N$, and also to $\mathcal{P}^+_{\lambda,\Lambda}(D^2u)\geq|Du|^\gamma$. Therefore, by Lemma \ref{mbound} we have
 \begin{equation}\label{eq2134}
 m(R)\leq CR^{-\frac{2}{q-1}}.
 \end{equation}
 Note that this implies $m(R) \to 0$ as $R \to \infty$. Therefore, we can apply Lemma \ref{mRincr}, {\it ii)} to obtain for large $R > R_1$
 \[
 m(R) \ge m(R_1)R_1^{\nu} \, R^{-\nu} ,
 \]
 contradicting \eqref{eq2134} when $\nu < \frac2{q-1}$ and $R \to \infty$.

\end{proof}

Some remarks and open questions are now in order.

\begin{rem}\label{H1fail}
The result in Theorem \ref{mainsum} is sharp. Indeed, when $\gamma>\frac{\beta}{\beta-1}$ and $q>\frac{\beta}{\beta-2}$, one can easily construct a counterexample to the Liouville property for the equation
\[
\mathcal{P}^+_{\lambda,\Lambda}(D^2u)\geq u^q+ |Du|^\gamma\text{ in }\R^N
\]
of the form $v^\delta(x)=K_\delta(1+ |x|^2)^{-\frac{\delta}{2}}$ for some suitable $K_\delta>0$ provided that
\begin{equation}\label{assuz}
\max\left\{\frac{2}{q-1},\frac{2-\gamma}{\gamma-1}\right\}<\delta<\beta-2\ .
\end{equation}
Indeed, by Lemma \ref{prepP} we have
\begin{multline*}
\mathcal{P}^+_{\lambda,\Lambda}(D^2v^\delta)- (v^\delta)^q- |Dv^{\delta}|^\gamma\geq(1+|x|^2)^{-\frac{\delta}{2}-1}\big[ \lambda\delta K_\delta(\beta-\delta-2)- \\
-(K_\delta)^q(1+|x|^2)^{-\frac{\delta}{2} q + \frac{\delta}{2}+1}-(\delta K_\delta)^\gamma|x|^\gamma(1+|x|^2)^{-(\frac{\delta}{2}+1)\gamma + \frac{\delta}{2} + 1 }\big]\ ,
\end{multline*}
which is nonnegative whenever the inequalities in \eqref{assuz} are satisfied and $K_\delta$ is small enough.\\
Indeed, when $-\frac{\delta}{2} q + \frac{\delta}{2}+1<0$, i.e. $\delta>\frac{2}{q-1}$, we have
\[
-(K_\delta)^q(1+|x|^2)^{-\frac{\delta}{2} q + \frac{\delta}{2}+1}\geq -K_\delta^q.
\]
Moreover, using that $|x|^\gamma\leq (1+|x|^2)^\frac\gamma2$, we have $-(\frac{\delta}{2}+\frac12)\gamma + \frac{\delta}{2} + 1<0$ choosing $\delta>\frac{2-\gamma}{\gamma-1}$ and hence
\[
-(\delta K_\delta)^\gamma|x|^\gamma(1+|x|^2)^{-(\frac{\delta}{2}+1)\gamma + \frac{\delta}{2} + 1 }\geq -(\delta K_\delta)^\gamma (1+|x|^2)^{-(\frac{\delta}{2}+\frac12)\gamma + \frac{\delta}{2} + 1 }\geq -(\delta K_\delta)^\gamma.
\]
We are then left with the inequality
\begin{align*}
\mathcal{P}^+_{\lambda,\Lambda}(D^2v^\delta)- (v^\delta)^q&- |Dv^{\delta}|^\gamma\geq (1+|x|^2)^{-\frac{\delta}{2}-1}[\lambda\delta K_\delta(\beta-\delta-2)-(K_\delta)^q-(\delta K_\delta)^\gamma]\\
&=  K_\delta (1+|x|^2)^{-\frac{\delta}{2}-1}[\lambda\delta(\beta-\delta-2)-K_\delta^{q-1}-\delta^\gamma K_\delta^{\gamma-1}]\\
&\geq K_\delta (1+|x|^2)^{-\frac{\delta}{2}-1}[\lambda\delta(\beta-\delta-2)2-\max\{K_\delta^{q-1},\delta^\gamma K_\delta^{\gamma-1}\}].
\end{align*}
Therefore, if choose $\delta<\beta-2$ we are free to take $K_\delta$ small enough so that the quantity $\lambda\delta(\beta-\delta-2)-2\max\{K_\delta^{q-1},\delta^\gamma K_\delta^{\gamma-1}\}\geq0$.
\end{rem}

\begin{rem}\label{M-} The previous remark shows that the theorem is sharp whenever $F = \mathcal{P}^+_{\lambda,\Lambda}$. In turn, for other operators, one could have the same Liouville property for a wider range of exponents $q, \gamma$. The most favorable case is indeed when $F = \mathcal{P}^-_{\lambda,\Lambda}$. With this choice, one could first deduce the Liouville property for classical solutions from linear results involving the Laplacian, that have been obtained in \cite{QuaasJMPA,CaristiMitidieri}: since
$$
-\Lambda\Delta u \geq \mathcal{P}^-_{\lambda,\Lambda}(D^2u)\geq u^q+ |Du|^\gamma,
$$
then $u$ has to be constant whenever
\[
q>\frac{N}{N-2} \text{\, and \, } \gamma\leq\frac{N}{N-1} , \quad \text{(or\, $q\le \frac{N}{N-2}$ and any $\gamma>0$)} \ ,
\]
where $N>2$. However, this range can be further improved. Indeed, if $\mathcal{P}^-_{\lambda,\Lambda}(D^2u)\geq0$, arguing as in Lemma \ref{mRincr} one can prove that $R\mapsto m(R)R^{\alpha-2}$ is nondecreasing, where $\alpha=\frac{\lambda}{\Lambda}(N-1)+1$. Then, one can retrace the same proof of Theorem \ref{mainsum} replacing $\beta$ with $\alpha$, and obtain the result when 
\[
q>\frac{\alpha}{\alpha-2} \text{\, and \, } \gamma\leq\frac{\alpha}{\alpha-1} , \quad {\text{(or\, $q\le \frac{\alpha}{\alpha-2}$ and any $\gamma>0$)}} \ ,
\]
which is even a larger region for $(q,\gamma)$ since $\frac{\alpha}{\alpha-1} > \frac{N}{N-1} > \frac{\beta}{\beta-1} $ (at least when $\Lambda>\lambda$).

\end{rem}


\begin{rem}
One can generalize Theorem \ref{mainsum} to fully nonlinear equations of the form
\[
F(x,D^2u)\geq H(x,u,Du)\text{ in }\R^N\ ,
\]
with
\[
H(x,u,Du)\geq a(x)u^q+b(x)|Du|^\gamma
\]
where e.g. $a(x)=|x|^{-\sigma}$, $\sigma>-2$, and $b(x)\geq b_0$. In this case the range of exponents in Theorem \ref{mainsum} becomes $\gamma\leq \frac{\beta}{\beta-1}$ and $q>\frac{\beta+\sigma}{\beta-2}$ (the case $q\leq \frac{\beta+\sigma}{\beta-2}$ being a consequence of \cite[Theorem 4.1]{CLeoni}).\\
One can even replace $u^q$ with $f=f(u)$, where $f:(0,\infty)\to\R$ is a continuous function verifying
\[
\liminf_{s\to0}\frac{f(s)}{s^q}>0\ .
\]
Finally, we believe that the result in Theorem \ref{mainsum} can be generalized to exterior domains of $\R^N$, i.e. when the equation is posed on $\R^N\backslash B$ where $B\subset\R^N$ is any ball, $N>2$, following a similar scheme of proof. Nevertheless, some arguments become more delicate; for example, it is no longer true in general that $\min_{R_0 \le |x|\leq R}u(x)=\min_{|x|=R}u(x)$ on exterior domains. This will be the matter of a future research, together with the treatment of a general case, as summarized in the next open problem.
\end{rem}

\begin{opp}
Do the results in Theorem \ref{mainsum} extend to general Bellman-Isaacs equations as in \cite{ArmstrongCPAM}, that is, of the form $\sup_\sigma\inf_\eta \mathcal{L}^{\sigma\eta}u=0$, where $\mathcal{L}^{\sigma\eta}$ is a family of linear uniformly elliptic operators, and to exterior domains? What happens when the gradient term in \eqref{Fgrad+0} has the opposite sign as in \cite{QuaasJMPA}, i.e. $H(u,Du)=u^q-|Du|^\gamma$?
\end{opp}

Similar Liouville properties to those outlined in this section have been recently addressed in \cite{Barrios} for problems driven by fractional Laplacians. We now formulate the following
\begin{opp}
Do the results of this section extend to problems driven by fully nonlinear nonlocal operators as in \cite{FQadv}, namely for these problems involving
\[
\mathcal{P}^+u=\sup_{K\in\mathcal{F}}-\int_{\R^N}(u(x+y)+u(x-y)-2u(x))K(y)\,dy\ ,\\
\]
\[
\mathcal{P}^-u=\inf_{K\in\mathcal{F}}-\int_{\R^N}(u(x+y)+u(x-y)-2u(x))K(y)\,dy\ ,\\
\]
where $\mathcal{F}=\left\{K:\frac{\lambda}{|y|^{N+2s}}\leq K(y)\leq\frac{\Lambda}{|y|^{N+2s}}\ ,K(y)=K(-y) \right\}$, $\Lambda\geq\lambda>0$, $s\in(0,1)$?
\end{opp}

\section{Failure of the Liouville property for nonlinearities involving the product of powers of the unknown and its gradient}\label{sec;conj}
We address in this section the failure of the Liouville property for inequalities involving {\it products} of zero-th and first order nonlinearities, that is,
\[
F(x,D^2u)\geq u^q|Du|^\gamma \quad \text{ in }\R^N, \ q\geq0\ ,\gamma>1\ .
\]
with $F$ continuous, uniformly elliptic and satisfying $F(x,0)=0$. Nonlinearities of the form $u^q|Du|^\gamma$ appear, for example, after a change of variable in the generalized porous-medium equations (see e.g. the introduction in \cite{Chang}). By the uniform ellipticity, any solution to the above inequality satisfies in the viscosity sense
\begin{equation}\label{P+graduq}
\mathcal{P}^+_{\lambda,\Lambda}(D^2u)\geq u^q|Du|^\gamma \quad \text{ in }\R^N, \ q\geq0\ ,\gamma>1\ .
\end{equation}
The validity of the Liouville property in the linear case $F=-\Delta$ has been shown in  \cite{CaristiMitidieri, MelianDCDS}, and requires $(N-1)\gamma+(N-2)q\leq N$. For the nonlinear counterpart, we state the following necessary condition.
\begin{thm}\label{countexq}
Let $N>2$. Then, there exist non-constant nonnegative classical solutions to the inequality \eqref{P+graduq} when $(\beta-1)\gamma+(\beta-2)q>\beta$, where $\beta=\frac{\Lambda}{\lambda}(N-1)+1$. In particular, when $q=0$ there exist non-constant entire solutions to $\mathcal{P}^+_{\lambda,\Lambda}(D^2u)\geq |Du|^\gamma$ in $\R^N$ when $\gamma>\frac{\beta}{\beta-1}$.
\end{thm}

\begin{proof}[Proof of Theorem \ref{countexq}]

Consider $v(x)=C_\delta(1+|x|^2)^{-\frac{\delta}{2}}$ for $\delta,C_\delta>0$ to be determined. 
By Lemma \ref{prepP} we have
\[
\mathcal{P}^+_{\lambda,\Lambda}(D^2v)-v^q|Dv|^\gamma\geq C_\delta\delta\lambda \frac{\beta-\delta-2}{(1+|x|^2)^{\frac{\delta}{2}+1}}-\frac{(C_\delta\delta)^\gamma}{(1+|x|^2)^{(\frac{\delta}{2}+1)\gamma-\frac{\gamma}{2}}}\frac{C_\delta^q}{(1+|x|^2)^{\frac{\delta}{2}q}}
\]
The right-hand side of the above inequality is bigger than or equal to zero when
\[
\begin{cases}
\left(\frac{\delta}{2}+1\right)\leq\left(\frac{\delta}{2}+1\right)\gamma-\frac{\gamma}{2}+\frac{\delta}{2}q\\
\beta>\delta+2\\
\lambda\delta C_\delta(\beta-\delta-2)-(C_\delta\delta)^\gamma C_\delta^q\geq0\ .
\end{cases}
\]
In particular, coupling the first two constraints it is easy to see that one gets $(\beta-2)q+(\beta-1)\gamma>\beta$, while the fact that $\beta-\delta-2>0$ allows to choose a suitable (small) positive constant $C_\delta>0$. For example, when $q=0$ the equality in the first constraint is attained e.g. for $\delta=\frac{2-\gamma}{\gamma-1}$, while the second inequality gives $\gamma > \frac{\beta}{\beta-1}$, so that the function $v(x)=C_\delta (1+|x|^2)^{-\frac{2-\gamma}{2(\gamma-1)}}$ works as a counterexample for the Liouville property for $C_\delta$ small enough (so that the third inequality is satisfied).
\end{proof}

\begin{rem}
The same kind of function serves as a counterexample for equations posed on exterior domains of $\R^N$. It is straightforward to check that the (singular) function $v(x)=c|x|^{-\frac{2-\gamma}{\gamma-1}}$ in $\R^N\backslash\{0\}$ is a classical solution to $\mathcal{P}^+_{\lambda,\Lambda}(D^2v)\geq |Dv|^\gamma$ in $\R^N\backslash \{0\}$ for $\gamma>\frac{\beta}{\beta-1}$.
\end{rem}

We now formulate the following conjecture, that extends \cite[Theorem 7.1]{CaristiMitidieri}, \cite[Corollary 3]{MelianDCDS} and \cite[Theorem 2.1]{VeronDuke} to the fully nonlinear setting.
\begin{conj}\label{conjprod}
Any nonnegative solution to $\mathcal{P}^+_{\lambda,\Lambda}(D^2u)\geq u^q|Du|^\gamma$ in $\R^N$ is constant provided that $(\beta-2)q+(\beta-1)\gamma\leq \beta$, $\gamma>1$, $q\geq0$. This would agree with the linear case, since $\beta\to N$ when $\lambda / \Lambda \to 1$. 

Note that by a suitable change of variables, the problem can be reduced to the case $q = 0$. The heuristic calculation goes as follows (we proceed as in \cite[Theorem 2.1]{VeronDuke}): under the assumption
\begin{equation}\label{bqg}
(\beta-2)q+(\beta-1)\gamma<\beta
\end{equation}
we let $u=v^b$, for some $b\in\R$ satisfying $b(b-1)>0$ to be later determined. Moreover, to simplify the calculation, suppose that $u,v$ are smooth functions. We have
\[
Du=bv^{b-1}Dv\ ,|Du|^2=b^2v^{2(b-1)}|Dv|^2\ ,|Du|^\gamma=b^\gamma v^{\gamma(b-1)}|Dv|^\gamma
\]
and
\[
D^2u=b(b-1)v^{b-2}Dv\otimes Dv+bv^{b-1}D^2v\ .
\]
We then conclude by the inequality $\mathcal{P}^+_{\lambda,\Lambda}(M+N)\leq \mathcal{P}^+_{\lambda,\Lambda}(M)+\mathcal{P}^+_{\lambda,\Lambda}(N)$ that
\begin{align*}
\mathcal{P}^+_{\lambda,\Lambda}(D^2u)&=\mathcal{P}^+_{\lambda,\Lambda}(b(b-1)v^{b-2}Dv\otimes Dv+bv^{b-1}D^2v)\\
&\leq \mathcal{P}^+_{\lambda,\Lambda}(b(b-1)v^{b-2}Dv\otimes Dv)+\mathcal{P}^+_{\lambda,\Lambda}(bv^{b-1}D^2v)\\
&=-b(b-1)v^{b-2}\lambda |Dv|^2+bv^{b-1}\mathcal{P}^+_{\lambda,\Lambda}(D^2v)\ .
\end{align*}
From the inequality $\mathcal{P}^+_{\lambda,\Lambda}(D^2u)\geq u^q|Du|^\gamma$ we then get
\[
b\mathcal{P}^+_{\lambda,\Lambda}(D^2v)\geq \lambda b(b-1)\frac{|Dv|^2}{v}+b^\gamma v^s|Dv|^\gamma
\]
for $s=1-\gamma+b(q+\gamma-1)$. We now use the generalized Young's inequality ($s$ will be positive, see below) to find that
\[
|Dv|^{\frac{2s+\gamma}{s+1}}=\left(\frac{|Dv|^2}{v}\right)^\frac{s}{s+1}v^{\frac{s}{s+1}}|Dv|^{\frac{\gamma}{s+1}}\leq \eps^{\frac{s+1}{s}}\frac{|Dv|^2}{v}+\frac{1}{\eps^{s+1}}|Dv|^\gamma v^s
\]
for any $\eps>0$. We choose $\eps^{-1-s}=b^q$ to find that for $c>0$ satisfying $c^{\frac{s+1}{s}}=\lambda b^{1-\frac{\gamma}{s}}(b-1)$ it holds
\begin{equation}\label{reduced}
\mathcal{P}^+_{\lambda,\Lambda}(D^2v)\geq c|Dv|^z\text{ in }\R^N\ ,
\end{equation}
where
\[
z=\frac{2s+\gamma}{s+1}=\frac{2-\gamma+2b(q+\gamma-1)}{2-\gamma+b(q+\gamma-1)}\ .
\]
For $\gamma+q-1>0$, we take $b=1+\delta$, $\delta>0$, so that $b(b-1)>0$. Then, we observe that $s=\gamma+\delta(\gamma+q-1)>0$ and $s=1-\gamma+\delta(q+\gamma-1)>1-\gamma$, which implies $z>1$. By the assumption \eqref{bqg} we can take $\delta$ small enough to find $(\beta-2)s+(\beta-1)\gamma<\beta$. Indeed
\begin{multline*}
(\beta-2)s+(\beta-1)\gamma<(\beta-2)s+\beta-(\beta-2)q\\ =(\beta-2)(s-q)+\beta=(\beta-2)(q+\gamma-1)(\delta-1)+\beta\ ,
\end{multline*}
which is equivalent to
\begin{equation}\label{condz}
z<\frac{\beta}{\beta-1}\ ,
\end{equation}
since
\begin{multline*}
(\beta-2)s+(\beta-1)\gamma<\beta\implies 2s(\beta-1)+(\beta-1)\gamma<\beta(s+1)\\
\implies (2s+\gamma)(\beta-1)<\beta(s+1)\ .
\end{multline*}

Therefore, the Liouville property for \eqref{P+graduq} reduces to the one for \eqref{reduced}, which we conjecture to be true whenever \eqref{condz} holds.

A similar transformation can be carried out for the equation $\mathcal{P}^-_{\lambda,\Lambda}(D^2u)\geq u^q|Du|^\gamma$, by exploiting the inequality $\mathcal{P}^-_{\lambda,\Lambda}(M+N)\leq \mathcal{P}^-_{\lambda,\Lambda}(M)+\mathcal{P}^+_{\lambda,\Lambda}(N)$. Then, condition \eqref{bqg} becomes $(\alpha-2)q+(\alpha-1)\gamma<\alpha$, yielding $z<\frac{\alpha}{\alpha-1}$.

\end{conj}

\begin{rem}
Following Remark \ref{M-}, when $F=\mathcal{P}^-$, one can obtain the Liouville property for nonnegative solutions to the viscosity inequality
\[
\mathcal{P}^-_{\lambda,\Lambda}(D^2u)\geq u^q|Du|^\gamma\text{ in }\R^N
\]
from the known results for the Laplacian \cite[Theorem 7.1]{CaristiMitidieri}, \cite[Theorem 2.1]{VeronDuke}, \cite[Corollary 3]{MelianDCDS} via the inequality $-\Lambda\Delta u\geq \mathcal{P}^-_{\lambda,\Lambda}(D^2u)$ when $(N-2)q+(N-1)\gamma\leq N$. However, in this case one should expect the Liouville property when $(\alpha-2)q+(\alpha-1)\gamma\leq \alpha$, $\alpha=\frac{\lambda}{\Lambda}(N-1)+1$, which yields a wider region than $(N-2)q+(N-1)\gamma\leq N$.
\end{rem}

\begin{rem}
It is worth remarking that related Liouville properties to those in this section were established in \cite[Theorem 3.1]{BirDem}, where it was proved that if $u\in C(\R^N)$ is a nonnegative viscosity solution (in an appropriate sense) of
\[
\mathcal{P}^+_{\lambda,\Lambda}(D^2u)|Du|^\eta\geq h(x)u^q\ ,
\]
where
\[
\eta>-1\ ,h(x)=h|x|^\sigma\text{ for $|x|$ large}, h>0\text{ and }\sigma>-(\eta+2)\ ,
\]
and 
\[
0<q\leq\frac{1+\sigma+(\eta+1)(\beta-1)}{\beta-2}\ ,
\]
then $u\equiv0$. This is accomplished by modifying the proof in Lemma \ref{mbound} to achieve the bound $m(R)R^{\beta-2}\leq CR^{\beta-2-\frac{\eta+2+\sigma}{q-(1+\eta)}}$ through a different choice of the test function $\varphi$, which leads to the desired contradiction since $\beta-2-\frac{\eta+2+\sigma}{q-(1+\eta)}<0$.\\
Therefore, if one takes $\sigma=0$, $\gamma:=-\eta<1$, it is possible to deduce the Liouville property for suitable viscosity solutions of the inequality
\[
\mathcal{P}^+_{\lambda,\Lambda}(D^2u)\geq u^q|Du|^\gamma\text{ in }\R^N
\]
in the following range
\[
(\beta-2)q+(\beta-1)\gamma\leq \beta\ , \, 0<\gamma<1\ ,\, q>0\ ,\, N>2\ .
\]
that however is outside the framework of the previous Conjecture \ref{conjprod}. In particular, one cannot deduce the result in the case $q=0$ (namely in the case of the absence of zero-th order terms), that remains at this stage an open problem even in the sub-linear regime of the gradient nonlinearity for equations driven by fully nonlinear operators. Finally, observe that a suitable modification of the scheme of proof in \cite[Theorem 3.1]{BirDem} leads to the improved Liouville property for the inequality
\[
\mathcal{P}^-_{\lambda,\Lambda}(D^2u)\geq u^q|Du|^\gamma\text{ in }\R^N\ ,u\geq0\ ,
\]
under the assumptions
\[
(\alpha-2)q+(\alpha-1)\gamma\leq \alpha\ ,\, 0<\gamma<1\ , \, q>0\ , \, N>2\ , \, \alpha=\frac{\lambda}{\Lambda}(N-1)+1\ .
\]
\end{rem}

\begin{rem}\label{rembir}
A partial result for $F=\mathcal{P}^-_{\lambda,\Lambda}$, without any restriction on $q,\gamma$ could be deduced exploiting that nonnegative solutions to $\mathcal{P}^-_{\lambda,\Lambda}(D^2u)\geq u^q|Du|^\gamma$ are also solutions to $\mathcal{P}^-_{\lambda,\Lambda}(D^2u)\geq0$ under the assumption $N\leq\frac{\Lambda}{\lambda}+1$ through \cite[Theorem 3.2]{CLeoni}. For general uniformly elliptic operators, the Liouville property holds for any $q,\gamma\geq0$ in the one dimensional case, see Remark \ref{low} below.
\end{rem}

The previous result could be also the starting point to address the following generalization of \cite{FilippucciJDE2011,SirakovInd} to systems of fully nonlinear inequalities with first-order terms having super-linear growth.
\begin{opp}
Does the Liouville property hold for the system
\[
\begin{cases}
F(x,D^2u)\geq v^{q_1}|Dv|^{\gamma_1}&\text{ in }\R^N\ ,\\
F(x,D^2v)\geq u^{q_2}|Du|^{\gamma_2}&\text{ in }\R^N\ ,
\end{cases}
\]
where $F$ is uniformly elliptic and the exponents $q_1,q_2,\gamma_1,\gamma_2$ are as in the previous sections?
\end{opp}
We now mention an open problem for degenerate operators:
\begin{opp}
Let $\X$ be a family of H\"ormander vector fields, $D_\X$ be the horizontal gradient in $\R^m$ and $(D^2_\X)^*$ the symmetrized horizontal Hessian in $\mathrm{Sym}_m$, $m\leq N$, (see Section \ref{sec:sol} for precise definitions). Do the aforementioned Liouville properties extend to degenerate problems such as
\[
\mathcal{P}^\pm_{\lambda,\Lambda}((D^2_\X u)^*)\pm u^q|D_\X u|^\gamma\geq u^p\text{ in }\R^N
\]
for some values of $q,\gamma,p$?
\end{opp}


\section{Fully nonlinear equations with superlinear gradient growth perturbed by drift terms}

We now consider Liouville properties for viscosity solutions to fully nonlinear equations of the form
\begin{equation}\label{H1H2}
F(x,D^2u) - b(x)\cdot Du=H(Du)\text{ in }\R^N,
\end{equation}
where $F$ satisfies the same assumptions of the previous sections and $b$ is a confining vector field. Our results will be concerned with Hamiltonians $H(Du)$ with power-like growth. We will show that, under suitable assumptions on $b$, the drift term is strong enough to provide the Liouville property even when $H$ has negative sign.


 We assume that $b$ is locally Lipschitz in $x$, namely for all $R>0$ there exists $K_R>0$ such that
\begin{equation}\label{b}
\sup_{|x|,|y|\leq R}|b(x)-b(y)|\leq K_R|x-y|\ .
\end{equation}
As mentioned in \cite{BC}, a simple example of vector field satisfying the assumptions in this section are perturbation of Ornstein-Uhlenbeck velocity fields with a lower-order perturbation such as
\[
b(x)=\nu(m-x)+h(x)\ ,\nu>0\ ,m\in\R^N\ ,\lim_{|x|\to\infty}\frac{h(x)\cdot x}{|x|^2}=0.
\]
The main result of the section is the following
\begin{thm}\label{rep}
Let $F$ be continuous, uniformly elliptic and such that $F(x,0)=0$, $H(Du)=A|Du|^\gamma$, and either
$$A > 0, \quad \text{and} \quad \gamma>0$$
or
$$A \in \mathbb R, \quad \text{and} \quad \gamma= 2. $$
Assume also that $b$ satisfies \eqref{b} and
\begin{equation}\label{drift}
\limsup_{|x|\to\infty}b(x)\cdot x < \lambda-\Lambda(N-1)\ .
\end{equation}

Let $u$ be a supersolution to \eqref{H1H2} bounded from below. Then $u$ is constant.
\end{thm}

The result shows that the presence of a suitable drift term forces the Liouville property, independently of the positive power $\gamma$ in the equation. We recall once more that when $H(Du)=A|Du|^\gamma$ and $b \equiv 0$, one needs $\gamma\leq \frac{N}{N-1}(<2)$, as shown in \cite[Theorem 7.1]{CaristiMitidieri}. Nevertheless, one can drop this assumption by requiring the velocity field $b$ to point toward the origin for large $|x|$, and having a sufficiently large inward component. Such a condition is by no means new (it can be found in e.g. \cite{BC} or \cite{BG_lio1} 
and references therein). The main contribution of Theorem \ref{rep} is that \eqref{drift} is strong enough to yield the Liouville property even when the super-linear gradient term has a {\it repulsive} behavior, for example when $H(Du)=-|Du|^2$, where the sign is negative.

The proof of our result will be based on the following crucial lemma, concerning the Liouville property for \eqref{H1H2} when $H \equiv 0$. We will rely on the existence of a sort of ``Lyapunov function'', which in our nonlinear setting reads as the existence of an explosive subsolution to $\mathcal{P}^+_{\lambda,\Lambda}(D^2u)-b(x)\cdot Du = 0$, see \eqref{lyapnl} below.
 While the argument might look different with respect to the one that has been used in previous sections, we are still controlling from below the behavior of $u$ with suitable subsolutions, which are somehow playing the role of the fundamental solutions in Section 2. Its proof builds upon strategies already appeared in e.g. \cite[Theorem 2.1]{BG_lio1} for classical solutions to linear equations and \cite[Theorem 2.1]{BC}, \cite[Proposition 3.1]{BG_lio1}. Note that 
 the one-side bound can be relaxed, cf Remark \ref{rel} below.

\begin{lemma}\label{lemdrift} Under the assumptions on $b$ of Theorem \ref{rep}, let $u$ be a supersolution to 
\[
E(x,Du,D^2u) := \, \mathcal{P}^+_{\lambda,\Lambda}(D^2u)-b(x)\cdot Du\geq0 \quad \text{ in }\R^N
\]
bounded from below. Then $u$ is constant.

\end{lemma}
The proof is based on the following idea: first, one proves that there exists a compact set $K\subset\R^N$ such that $\inf_{\R^N}u=\min_{K}u$. Then, the function $z=u-\min_Ku$ satisfies the viscosity inequality $E(x,Du,D^2u)\geq0$ in $\R^N$, while $z$ vanishes somewhere in $K$. Therefore, $z\equiv0$ by the strong minimum principle, and hence $u$ is constant. We prove this result through several steps, with the aid of a Lyapunov function, as outlined in the next proof.
We premise the following well-known property, known as \textit{transitivity of the viscosity inequalities}.
\begin{lemma}\label{transvisc}
Let $u$ be satisfying the inequality $\mathcal{P}^+_{\lambda,\Lambda}(D^2u)\geq f$ in the viscosity sense and $v$ be such that the inequality $\mathcal{P}^+_{\lambda,\Lambda}(D^2v)\leq g$ holds in the viscosity sense. Then, $w=u-\xi v$, $\xi>0$, satisfies $\mathcal{P}^+_{\lambda,\Lambda}(D^2w)\geq f-\xi g$ in the viscosity sense.  
\end{lemma}
\begin{proof}
The proof is simple when one of the involved functions is smooth, cf e.g. \cite[Lemma 2.12]{CC}. A more general statement can be found in \cite[Theorem 5.3]{CC} and \cite[p. 745]{ArmstrongCPAM}, while a proof for Hamilton-Jacobi-Bellman operators can be found in \cite[Theorem 2.1-Step 3]{BC} or \cite[Proposition 3.1]{BG_lio1}.
\end{proof}
\begin{proof} \textit{Step 1. (Existence of a Lyapunov function)} We first observe that there exists $\bar R\geq0$ such that the function $w(x)=-\log|x|$ satisfies $w\in C^2(\R^d\backslash\{0\})$, 
\begin{equation}\label{lyapnl}
\text{$\lim_{|x|\to\infty}w(x)=-\infty$ \ and \ $E(x,Du,D^2u)\leq0$ in \ $\R^N\backslash \overline{B}(0,\bar R)$}.
\end{equation}
 Indeed,
\[
\mathcal{P}^+_{\lambda,\Lambda}(D^2w)-b(x)\cdot Dw=\frac{\Lambda(N-1)}{|x|^2}-\frac{\lambda}{|x|^2}+b(x)\cdot \frac{x}{|x|^2}
\]
and \eqref{drift} implies the (last) claim outside a large ball of radius $\bar R$. 

\textit{Step 2}. For $\xi>0$, we set $v_\xi:=u-\xi w$. Then, $v_\xi$ is continuous on the set $\{x\in\R^N:|x|\geq \bar R\}$. Moreover, by Lemma \ref{transvisc}, we have that the following inequality holds in the viscosity sense
\[
E(x,Dv_\xi,D^2v_\xi)\geq0\text{ for every $x$ such that $|x|>\bar R$}.
\]
Define $c_\xi=\min_{|x|= \bar R}v_\xi$. Since $u$ is bounded below, we have
\[
\lim_{|x|\to\infty}v_\xi(x)=+\infty\ ,
\]
and hence there exists $K_\xi>\bar R$ such that $v_\xi>c_\xi$ for every $x$ such that $|x|\geq K_\xi$. \\

\textit{Step 3}. We use the comparison principle on $A_{R,K_\xi}:=\{x\in\R^N:\bar R<|x|<K_\xi\}$, cf \cite{BC,BG} and the references therein, to conclude
\[
\min_{A_{R,K_\xi}}v_\xi(x)=\min_{\{x\in\R^N:|x|=\bar R\text{ or }|x|=K_\xi\}}v_\xi(x)=c_\xi\ .
\]
Since $v_\xi>c_\xi$ for $|x|\geq K_\xi$, we get for all $|y|\geq \bar R$,
\[
v_\xi(y)\geq c_\xi=\min_{\{x\in\R^N: |x|=\bar R\}}u+\xi \max_{\{x\in\R^N:|x|=\bar R\}}w\ .
\]
We let $\xi\to0$ and conclude
\[
u(y)\geq \min_{\{x\in\R^N:|x|=\bar R\}}u\ ,|y|\geq\bar R\ .
\]
On the other hand, the comparison principle applied on the ball $B(0,\bar R)$ leads to
\[
u(y)\geq \min_{\{x\in\R^N:|x|=\bar R\}}u\ ,|y|<\bar R\ .
\]
Combining the above assertions, we conclude that $u$ attains its minimum at some point of $\partial B(0,\bar R)$, so that $u$ is constant by the strong minimum principle \cite{BG,GP}.
\end{proof}

 We can now prove the main theorem of this section.
 
\begin{proof}[Proof of Theorem \ref{H1H2}]
Let us start with the case $A > 0,  \gamma>0$. In such case, it is sufficient to observe that $u$ is a supersolution bounded from below to
\[
\mathcal{P}^+_{\lambda,\Lambda}(D^2u) - b(x)\cdot Du\geq A|Du|^\gamma \quad \text{ in }\R^N
\]
and in turn to
\[
E(x,Du,D^2u)\geq0\text{ in }\R^N\ ,
\]
so it is sufficient to apply Lemma \ref{lemdrift} to obtain the desired result.
\smallskip

Let us now consider the case $H(Du)=A |Du|^2$, and assume $A = -1$ for simplicity. Then $u$ solves in the viscosity sense
 \[
\mathcal{P}^+_{\lambda,\Lambda}(D^2u)+|Du|^2-b(x)\cdot Du\geq0\text{ in }\R^N\ .
\]
By the nonlinear Hopf-Cole transform, the function \[v(x)=\lambda(1-e^{-\frac{1}{\lambda}u(x)})\] is a bounded below viscosity solution to the inequality
\[
E(x,Dv,D^2v) = \mathcal{P}^+_{\lambda,\Lambda}(D^2v)-b(x)\cdot Dv\geq 0\text{ in }\R^N\ .
\]
To see this, assume first that $u$ is smooth. We have
\[
Dv=Du\ e^{-\frac{1}{\lambda}u},\ \quad D^2v=e^{-\frac{1}{\lambda}u}D^2u-\frac{1}{\lambda}Du\otimes Du\ e^{-\frac{1}{\lambda}u}\ .
\]
Then, exploiting the inequality $\mathcal{P}^+_{\lambda,\Lambda}(M+N)\geq \mathcal{P}^+_{\lambda,\Lambda}(M)+\mathcal{P}^-_{\lambda,\Lambda}(N)$, we get
\begin{align*}
\mathcal{P}^+_{\lambda,\Lambda}(D^2v)&-b(x)\cdot Dv=\mathcal{P}^+_{\lambda,\Lambda}\left(e^{-\frac{1}{\lambda}u}D^2u-\frac{1}{\lambda}Du\otimes Du\ e^{-\frac{1}{\lambda}u}\right)-b(x)\cdot Du\ e^{-\frac{1}{\lambda}u}\\
&\geq \mathcal{P}^+_{\lambda,\Lambda}(D^2u)e^{-\frac{1}{\lambda}u}+\mathcal{P}^-_{\lambda,\Lambda}\left(-\frac{1}{\lambda}Du\otimes Du\ e^{-\frac{1}{\lambda}u}\right)-b(x)\cdot Du\ e^{-\frac{1}{\lambda}u}\\
&=\mathcal{P}^+_{\lambda,\Lambda}(D^2u)e^{-\frac{1}{\lambda}u}-\frac{1}{\lambda}\mathcal{P}^+_{\lambda,\Lambda}\left(Du\otimes Du\right)e^{-\frac{1}{\lambda}u}-b(x)\cdot Du\ e^{-\frac{1}{\lambda}u}\\
&=(\mathcal{P}^+_{\lambda,\Lambda}(D^2u)+|Du|^2-b(x)\cdot Du)e^{-\frac{1}{\lambda}u}\ ,
\end{align*}
namely
\[
\frac{\mathcal{P}^+_{\lambda,\Lambda}(D^2v)-b(x)\cdot Dv}{1-v/\lambda}\geq \mathcal{P}^+_{\lambda,\Lambda}(D^2u)+|Du|^2-b(x)\cdot Du\ .
\]
When $u$ is just a viscosity supersolution, suppose by contradiction that $v$ does not satisfy $\mathcal{P}^+_{\lambda,\Lambda}(D^2v)-b(x)\cdot Dv\geq0$ in the viscosity sense. Then, there exists a smooth function $\psi$ such that $v-\psi$ attains a local minimum at some $x_0 \in A$, and
\[
\mathcal{P}^+_{\lambda,\Lambda}(D^2\psi)-b(x)\cdot D\psi<0\text{ in }A\ .
\]
We then define $\varphi(x)=-\lambda\log\left(1-\frac{\psi(x)}{\lambda}\right)$ in the above inequality to get a contradiction with the fact that $u$ solves in the viscosity sense the initial equation, and $u-\varphi$ attains a local minimum at $x_0$. Hence, we can again apply Lemma \ref{lemdrift}.
\end{proof}

\begin{rem}\label{rel}
One can relax the one-side bound from below on $u$ by the weaker growth condition $\liminf_{|x|\to\infty}\frac{u(x)}{\log|x|}\geq0$. Clearly, the latter condition is satisfied if $u$ is bounded below. 

Furthermore, one can also get a sufficient condition for the Liouville property in the presence of zero-th order terms such as $c(x)u$, $c>0$, by further imposing a two-side bound on the solution, see e.g. \cite{BG_lio1} or \cite{Pigola}, where the property is called $\lambda$-Liouville property. In the case $\lambda=\Lambda=1$, this implies that $\R^N$ is stochastically complete. The drift term can be also replaced with a more general concave Hamiltonian such as
\[
H_{\text{conc}}(x,u,Du)=\inf_{\alpha\in A}\{c^\alpha(x)u-b^\alpha(x)\cdot Du\}\ ,
\]
where $A$ is a metric space, under appropriate assumptions on $b^\alpha,c^\alpha$, cf \cite{BC,BG_lio1}. 

Note that the theorem can be extended in an obvious manner to non-negative Hamiltonians. On the other hand, the power-like case $H = A|Du|^\gamma$, with {\it negative} $A$ and $\gamma \neq 2$ is much more delicate. One may conjecture that the Liouville property holds as in the quadratic case $\gamma = 2$, but we do not know how to prove it at this stage.

We conclude by saying that some Liouville properties for different equations posed on Riemannian manifolds through similar methods have been obtained for solutions to linear PDEs satisfying two-side bounds in e.g. \cite[Section 13]{Gry1} and \cite{MPcomm}, see also \cite{PP} for fully nonlinear problems, where the existence of a Lyapunov function is required. Instead, other results for solutions bounded from one-side to nonlinear partial differential inequalities driven by the $p$-Laplacian have been established in \cite[Section 4]{Pigola} and recently in \cite[Theorem 1.2]{MV}.
\end{rem}

\begin{rem}
Other sufficient conditions implying the Liouville property can be obtained by considering different negative power-like functions, as $w(x)=-\frac{|x|^2}{2}$.
\end{rem}

\begin{rem}\label{sharpdrift}
 We show that the condition \eqref{drift} imposed on the velocity field is optimal for the Liouville property, at least for the class of operators $\mathcal P^+ +b(x)\cdot D$, with the velocity field $b$ satisfying the above assumptions. In the nonlinear setting, the claim that the validity of the Liouville property implies the existence of Lyapunov functions has been matter of recent research (see e.g. \cite{MPcomm} and references therein). Here, we obtain such implication (and therefore the equivalence) for a family of problems involving  $\mathcal P^+$ and Ornstein-Uhlenbeck drifts.

Note first that condition \eqref{drift} can be equivalently rewritten as
\[
\limsup_{|x|\to\infty}\{b(x)\cdot x\}<\lambda(2-\beta)\ .
\]
Then, for $\delta>0$ we take $u(x)=(1+|x|^2)^{-\frac{\delta}{2}}=:f(|x|)$. Computations in Lemma \ref{prepP} with $C_\delta=1$ give (using that $\Lambda(N-1)=\lambda(\beta-1)$ and the fact that for $|x|^2<1/(\delta+1)$ one has $\Lambda[1-(\delta+1)|x|^2]>\lambda[1-(\delta+1)|x|^2]$)
\[
\mathcal{P}^+_{\lambda,\Lambda}(D^2u)\geq  \delta\lambda\left[\frac{\beta-1}{(1+|x|^2)^{\frac{\delta}{2}+1}}-\frac{(\delta+1)|x|^2-1}{(1+|x|^2)^{\frac{\delta}{2}+2}}\right]\ .
\]
Hence, for $b(x)=\lambda (2-\beta+\delta)\frac{x}{1+|x|^2}$, we have
\[
\mathcal{P}^+_{\lambda,\Lambda}(D^2u)-b(x)\cdot Du=\beta\delta\lambda(1+|x|^2)^{-\frac{\delta}{2}-2}\geq0\text{ in }\R^N\ ,
\]
while $\lim_{|x|\to\infty}b(x)\cdot x>\lambda(2-\beta)$. 

Regarding the degenerate setting addressed in \cite{BG_lio1}, and the sharpness of the conditions therein, similar counterexamples can be built for PDEs over H\"ormander vector fields in the presence of a ``fundamental solution". This issue has been recently investigated in \cite{BG_lio2}.

\end{rem}

Lastly, we show how to improve the condition in Theorem \ref{rep} when looking at supersolutions to $F=\mathcal{P}^-_{\lambda,\Lambda}$ instead of $F=\mathcal{P}^+_{\lambda,\Lambda}$ (as we did in the previous Remark \ref{M-}). Here, the crucial point is that the difference between the supersolution and the Lyapunov function $w$, which is now a classical subsolution to $\mathcal{P}^-_{\lambda,\Lambda}(D^2u)-b \cdot Du$, turns out to be a supersolution of a problem involving $\mathcal{P}^+_{\lambda,\Lambda}$. That $\mathcal{P}^-$ switches into $\mathcal{P}^+$ is not a problem, since we still have the strong minimum principle and the comparison principle on bounded open sets.
\begin{thm}\label{rep2}
Let $u$ be a viscosity supersolution bounded below to 
\[
\mathcal{P}^-_{\lambda,\Lambda}(D^2u)-b(x)\cdot Du=A |Du|^\gamma \text{ in }\R^N,
\]
and either
$$A > 0, \quad \text{and} \quad \gamma>0$$
or
$$A \in \mathbb R, \quad \text{and} \quad \gamma= 2. $$
Assume also that $b$ satisfies \eqref{b} and
\begin{equation}\label{drift2}
\limsup_{|x|\to\infty}\{b(x)\cdot x\}<\Lambda-\lambda(N-1)\ .
\end{equation}
Then, $u$ is constant.
\end{thm}

\begin{rem}
Note that condition \eqref{drift2} is better than \eqref{drift} since 
\[
\Lambda-\lambda(N-1)>\lambda-\Lambda(N-1)\iff N(\Lambda-\lambda)>0\ .
\] 
We further remark that in such case the Liouville property could have been deduced via the corresponding results for the Laplacian since $\mathcal{P}^-_{\lambda,\Lambda}(D^2u)\leq-\Lambda\Delta u$, where the condition would read as
\begin{equation}\label{drift3}
\limsup_{|x|\to\infty}\{b(x)\cdot x\}<\Lambda(2-N)\ .
\end{equation}
However, condition \eqref{drift2} is again better than \eqref{drift3}, since
\[
\Lambda-\lambda(N-1)>\Lambda(2-N)\iff (\Lambda-\lambda)(N-1)>0\ .
\]
\end{rem}

\begin{proof}[Proof of Theorem \ref{rep2}]
The proof is similar to Theorem \ref{rep}. We only outline the main differences. The Liouville property for the viscosity inequality $\mathcal{P}^-_{\lambda,\Lambda}(D^2u)+b(x)\cdot Du\geq0$ can be obtained exactly as in Lemma \ref{lemdrift}, using $w(x)=-\log|x|$ as a Lyapunov function (an explosive subsolution of $\mathcal{P}^-_{\lambda,\Lambda}(D^2u)-b \cdot Du$). Note that for $\xi>0$, by a standard test function argument together with the inequality $\mathcal{P}^+_{\lambda,\Lambda}(M+N)\geq \mathcal{P}^+_{\lambda,\Lambda}(M)+\mathcal{P}^-_{\lambda,\Lambda}(N)$, one verifies that $v_\xi = u-\xi w$ solves in viscosity sense
\[
\mathcal{P}^+_{\lambda,\Lambda}(D^2v_\xi)-b(x)\cdot Dv_\xi\geq0\text{ for every $x$ such that $|x|>\bar R$},
\]
namely it is a supersolution of an equation involving the maximal operator $\mathcal{P}^+$ (and not the minimal operator $\mathcal{P}^-$). Still, one has the strong minimum principle and the comparison principle on the annuli to achieve the result.

 In the case $H(Du)=-|Du|^2$, one argues as in Theorem \ref{rep}, now using the inequality $\mathcal{P}^-_{\lambda,\Lambda}(M+N)\geq \mathcal{P}^-_{\lambda,\Lambda}(M)+\mathcal{P}^-_{\lambda,\Lambda}(N)$ to find that for $v(x)=\lambda(1-e^{-\frac{1}{\lambda}u(x)})$ and smooth $u$
\begin{align*}
\mathcal{P}^-_{\lambda,\Lambda}(D^2v)&-b(x)\cdot Dv=\mathcal{P}^-_{\lambda,\Lambda}\left(e^{-\frac{1}{\lambda}u}D^2u-\frac{1}{\lambda}Du\otimes Du\ e^{-\frac{1}{\lambda}u}\right)-b(x)\cdot Du\ e^{-\frac{1}{\lambda}u}\\
&\geq \mathcal{P}^-_{\lambda,\Lambda}(D^2u)e^{-\frac{1}{\lambda}u}+\mathcal{P}^-_{\lambda,\Lambda}\left(-\frac{1}{\lambda}Du\otimes Du\ e^{-\frac{1}{\lambda}u}\right)-b(x)\cdot Du\ e^{-\frac{1}{\lambda}u}\\
&=\mathcal{P}^-_{\lambda,\Lambda}(D^2u)e^{-\frac{1}{\lambda}u}-\frac{1}{\lambda}\mathcal{P}^+_{\lambda,\Lambda}\left(Du\otimes Du\right)e^{-\frac{1}{\lambda}u}-b(x)\cdot Du\ e^{-\frac{1}{\lambda}u}\\
&=(\mathcal{P}^-_{\lambda,\Lambda}(D^2u)+|Du|^2-b(x)\cdot Du)e^{-\frac{1}{\lambda}u} \ge 0\ .
\end{align*}

\end{proof}


\begin{rem}\label{mixquad}
As a final suggestion, following \cite[Theorem 1.1]{Chang}, a Liouville result as in Theorems \ref{rep} and \ref{rep2} can be obtained for the problem
\[
\mathcal{P}^+_{\lambda,\Lambda}(D^2u)-b(x)\cdot Du+u^q|Du|^2\geq0\text{ in }\R^N\ ,q\geq0
\]
via the transformation
\[
v(x)=\int_0^{u(x)}e^{-\frac{s^{q+1}}{(q+1)\lambda}}\,ds\ .
\]
Indeed, assuming that $u$ is a smooth function, standard calculations lead to
\[
Dv=e^{-\frac{u^{q+1}}{\lambda(q+1)}}Du\ ;
\]
\[
D^2v=-e^{-\frac{u^{q+1}}{\lambda(q+1)}}\frac{u^q}{\lambda}Du\otimes Du+e^{-\frac{u^{q+1}}{\lambda(q+1)}}D^2u\ .
\]
Then,
\begin{multline*}
\mathcal{P}^+_{\lambda,\Lambda}(D^2v)-b(x)\cdot Dv=\mathcal{P}^+_{\lambda,\Lambda}\left(-e^{-\frac{u^{q+1}}{\lambda(q+1)}}\frac{u^q}{\lambda}Du\otimes Du+e^{-\frac{u^{q+1}}{\lambda(q+1)}}D^2u\right)\\
-e^{-\frac{u^{q+1}}{\lambda(q+1)}}b(x)\cdot Du\\
\geq \mathcal{P}^-_{\lambda,\Lambda}\left(-e^{-\frac{u^{q+1}}{\lambda(q+1)}}\frac{u^q}{\lambda}Du\otimes Du\right)+\mathcal{P}^+_{\lambda,\Lambda}\left(e^{-\frac{u^{q+1}}{\lambda(q+1)}}D^2u\right)-e^{-\frac{u^{q+1}}{\lambda(q+1)}}b(x)\cdot Du\\
=e^{-\frac{u^{q+1}}{\lambda(q+1)}}\left(u^q|Du|^2+\mathcal{P}^+_{\lambda,\Lambda}(D^2u)-b(x)\cdot Du\right)\geq0\ ,
\end{multline*}
so that the Liouville property for $u$ boils down to the one for $v$.

A similar transformation can be applied to problems driven by the singular operator \[F(Du,D^2u)=|Du|^{m-2}\mathcal{P}^+_{\lambda,\Lambda}(D^2u),\] $m>1$, and nonlinearities of the form $H(u,Du)=u^q|Du|^\gamma$ with $m=\gamma$ and any $q\geq0$, as in \cite{Chang} for the case of the $m$-Laplacian.
\end{rem}

\section{Further remarks}\label{sec:rem}
We now discuss how to derive the Liouville property in some special situations.
\begin{rem}\label{nondiv}
It is immediate to conclude that all the Liouville-type results obtained above lead to new Liouville properties for supersolutions to equations driven by non-divergence operators of the form
\[
-\mathrm{Tr}(A(x, u, Du)D^2u)\geq H_i(x,u,Du)\text{ in }\R^N\ ,
\]
where $\lambda I_N\leq A\leq \Lambda I_N$ (see for example Remark \ref{nplapl} below). We emphasize that the literature on Liouville properties for such linear operators is not as wide as the one on constant coefficient equations. See for example \cite{Sobol} for some results on problems with zero-th order terms. All the known properties in the non-divergence setting with the presence of gradient dependent nonlinearities have been actually obtained from those on fully nonlinear problems.
\end{rem}

\begin{rem}\label{low} In Sections \ref{sec;sum} and \ref{sec;conj} we analyzed Liouville-type results for $\beta>2$ and $N>2$. The cases $N=1$ (that corresponds to $\beta \leq2$) and $N=2$ are in general simpler to address, as it happens for quasi-linear problems driven by the Laplacian, since the Liouville properties follow from the corresponding ones of superharmonic functions.

First, when $\beta\leq2$, the Liouville properties for (classical) solutions of $F(x,D^2u)\geq H_i(x,u,Du)$ follow from the fact that there are no non-nonstant positive concave functions on the real line, since when $F(x,0)=0$ and $F$ is uniformly elliptic, it holds $\mathcal{P}^+_{\lambda,\Lambda}(u'')\geq F(x,u'')$ and so $u$ solves $\mathcal{P}^+_{\lambda,\Lambda}(u'')\geq0$ with
\[
\mathcal{P}^+_{\lambda,\Lambda}(s)=\begin{cases}
-\lambda s\text{ if }s\geq0,\\
-\Lambda s\text{ if }s<0.
\end{cases}
\]
When instead $N=2$ and $F=\mathcal{P}^-_{\lambda,\Lambda}$, the result of Theorem \ref{mainsum} follows from the fact that the supersolutions to $F\geq H$ are superharmonic, so $u$ is constant by e.g. \cite[Theorem 2.1]{BG_lio1} or \cite[Theorem 29 Ch. 2]{PW}, and it holds for any $q,\gamma\geq0$. Similarly, Conjecture \ref{conjprod} for $F=\mathcal{P}^-_{\lambda,\Lambda}$ is true by the same argument for any $q,\gamma\geq0$ in $\R^2$. The case of arbitrary fully nonlinear operators in the plane (and in particular when $F=\mathcal{P}^+_{\lambda,\Lambda}$) is more delicate, since functions are no longer superharmonic in $\R^2$. If one exploits the comparison with the maximal operator, the condition ensuring the Liouville property becomes $N\leq \frac{\lambda}{\Lambda}+1$ (see \eqref{drift} with $b=0$), which rules out the case $N=2$. Therefore, even in the non-divergent case, the Liouville property in the plane is in general lost if no further conditions are imposed. Indeed, if one considers the equation $-\mathrm{Tr}(AD^2u)\geq H$, since $H\geq0$, it follows that $u$ is a solution to $-\mathrm{Tr}(AD^2u)\geq0$ in $\R^2$, and a Liouville property could be obtained under further assumptions at infinity of the diffusion matrix, cf e.g. \cite{Sobol}. For instance, one can apply \cite[Theorem 2.1]{BG_lio1} provided there exists a Lyapunov function for $-\mathrm{Tr}(A(x)D^2w)=0$ in $\R^2$, $\lambda I_N\leq A\leq \Lambda I_N$. Standard calculations for a radial function $w(x)=f(|x|)$ lead to
\[
-\mathrm{Tr}(A(x)D^2w)=-\frac{\mathrm{Tr}(Ax\otimes x)}{|x|^2}\left(f''(|x|)+\frac{\Psi_a(x)-1}{|x|}f'(|x|)\right)\ ,
\]
where $\Psi_a(x)=\frac{\mathrm{Tr}(A)}{\frac{\mathrm{Tr}(Ax\otimes x)}{|x|^2}}$ is the so-called effective dimension of the linear operator, cf \cite[p.518]{MeyersSerrin}, which is now the quantity responsible for the qualitative behavior of the linear operator. Therefore, if one takes $w(x)=-\log|x|$, one has
\[
-\mathrm{Tr}(A(x)D^2w)=\frac{\mathrm{Tr}(Ax\otimes x)}{|x|^2}\left(\frac{\Psi_a(x)-2}{|x|^2}\right)\ .
\]
Note that for $A=I_N$ it follows $\Psi_a(x)=N$, while $\Psi_a$ stands for the linear counterpart of the intrinsic dimension $\beta$ of the Pucci's maximal operator. Hence, when $A=I_2$ it follows that $w$ is a subsolution in the plane, and the Liouville property holds for the Laplace equation in $\R^2$, cf \cite[Remark 2.2]{BG_lio1}. For general uniformly elliptic operators in non-divergence form, heuristically, if the matrix of the coefficients approaches to a constant matrix for large $|x|$, the value responsible for the Liouville properties of the non-divergent equation would be the same as the one for the Laplacian.

\end{rem}

\begin{rem} Owing to the same techniques and with appropriate modifications, the symmetric Liouville property for \textit{subsolutions bounded above} to the viscosity inequality
\[
F(x,D^2u)\le H_i(x,u,Du) \quad \text{ in }\R^N
\]
can be obtained via the analysis of the partial differential inequality
\[
\mathcal{P}^-_{\lambda,\Lambda}(D^2u)\le H_i(x,u,Du) \quad \text{ in }\R^N\ .
\]
\end{rem}

\begin{rem} We have discussed in this paper the Liouville properties for $C$-viscosity inequalities involving fully nonlinear operators. As far as we know, the analysis of similar properties in the framework of $L^p$-viscosity solutions is widely open. However, some results can be obtained in a straightforward manner from the linear theory. For example, if one looks at nonpositive (and not nonnegative!) $L^p$-viscosity solutions to the inequality $\mathcal{P}^-_{\lambda,\Lambda}(D^2u)\geq|Du|^\gamma$, it is possible to prove that $u$ is a constant for every $\gamma>0$ by requiring also $u\in L^p(\R^N)$, $1<p<\infty$. Indeed, $u$ solves also the inequality $-\Lambda\Delta u\geq |Du|^\gamma$ and in particular $u$ is superharmonic.  Then $v=-u$ is subharmonic and nonnegative. Combining the mean-value property and the H\"older's inequality one concludes that $v$ is a constant, and so is $u$. 
\end{rem}
\begin{rem}\label{quadlio}
It is well-known that the classical Liouville property fails for non-homogeneous problems, a simple example being the function $u(x)=|x|^2$, which is a nonnegative and non-constant solution to $-\Delta u=k$ in $\R^N$ with $k=-2N$. 
Therefore, if one avoids the restriction $F(x,0)=0$, and the equation is $(u, Du)$-independent, it is possible to formulate and expect different Liouville properties. For instance, if one imposes the condition $F(x,0)=k$ for some constant $k\in\R$ in the fully nonlinear case, one can expect to prove that $u$ is a quadratic polynomial instead of a constant. An instance of such a Liouville phenomenon appears for the Monge-Amp\`ere equation as a consequence of second order bounds, cf \cite[Theorem 1.7.2]{NV} or \cite[Remark 4]{Mooney}, the so-called J\"orgen-Calabi-Pogorelov Theorem, the Special Lagrangian equation \cite[Theorem 1.7.3]{NV} and other Hessian equations, as discussed in e.g. \cite[Corollary 1.2]{MS} and \cite{NV}.\\
We conclude by saying that in the case of solutions to Hessian equations without $x$-dependence, a related assumption is crucial. For instance, when looking at solutions (not super- or subsolutions) to $F(D^2u)=0$ in $\R^N$, the assumption $F(0)=0$ is essential to deduce the Liouville property when $N\geq 5$, cf \cite[Section 1.7.1]{NV}. However, it was conjectured, see \cite[Conjecture 1.7.1]{NV}, that such a restriction can be dropped in the lower dimensional cases $N\leq 4$.
\end{rem}

\begin{rem}\label{nplapl}
The above results involving fully nonlinear uniformly elliptic equations can be extended to problems driven by the normalized $p$-Laplacian: for $p\in(1,\infty)$, the operator \[-\widetilde\Delta_p u=-\frac{1}{p}|Du|^{2-p}\mathrm{div}(|Du|^{p-2}Du)=-\frac{1}{p}\Delta u-\frac{p-2}{p}\frac{\Delta_\infty u}{|Du|^2},\] can be indeed compared with Pucci's extremal operators. Unlike the $p$-Laplacian, we can write the above operator in non-divergence form as
\[
-\widetilde\Delta_p u=-\mathrm{Tr}(A(Du)D^2u)
\]
where
\[
A(Du)=\frac1p\left(I_N+(p-2)\frac{Du\otimes Du}{|Du|^2}\right)\ ,
\]
which is positive definite, and bounded for every $p\in(1,\infty)$. Moreover, its eigenvalues are bounded below by $\lambda=\min\left\{\frac1p,\frac{p-1}{p}\right\}$ and $\Lambda=\max\left\{\frac1p,\frac{p-1}{p}\right\}$. We then have
\[
\mathcal{P}^-_{\lambda,\Lambda}(D^2u)\leq -\widetilde\Delta_p u\leq \mathcal{P}^+_{\lambda,\Lambda}(D^2u)
\]
whenever $\lambda=\min\left\{\frac1p,\frac{p-1}{p}\right\}$ and $\Lambda=\max\left\{\frac1p,\frac{p-1}{p}\right\}$. Therefore, we can state the following consequences of Theorems \ref{mainsum} and \ref{rep}.

\begin{cor}
Any nonnegative viscosity solution to 
\[
-\widetilde\Delta_p u\geq u^q+|Du|^\gamma\text{ in }\R^N\ ,
\]
with $p>2$, $1<\gamma\leq \frac{(p-1)(N-1)+1}{(p-1)(N-1)}$ and $q>\frac{(p-1)(N-1)+1}{(p-1)(N-1)-1}$ is constant. Instead, when $1<p<2$ the same assertion holds for $1<\gamma\leq \frac{N+p-2}{N-1}$ and $q>\frac{N+p-2}{N-p}$.
\end{cor}

\begin{cor}\label{gamedrift}
Let $u$ be a viscosity solution bounded from below to
\[
-\widetilde\Delta_p u+b(x)\cdot Du\geq A|Du|^\gamma \qquad \text{ in }\R^N\ ,
\]
and either
$$A > 0, \quad \text{and} \quad \gamma>0$$
or
$$A \in \mathbb R, \quad \text{and} \quad \gamma= 2. $$
Assume that $b$ satisfies \eqref{b} and
\begin{equation}
\limsup_{|x|\to\infty}b(x)\cdot x<
\begin{cases}
1-\frac{N(p-1)}{p} & \text{if $p > 2$} \\
1-\frac{N}{p}  & \text{if $1 < p \le 2$} .
\end{cases}
\end{equation}
Then, $u$ is constant.
\end{cor}

\end{rem}

We finally mention that, for nonlinear problems considered above, one could address different forms of the Liouville property. The following questions are inspired by the works by V. Kurta (see e.g. \cite{KurtaKawohl} and the references therein).

\begin{opp}[Liouville comparison principle]
It is known that if $u,v$ solve in weak sense the inequality
\[
-\Delta u-|u|^{q-1}u\leq -\Delta v-|v|^{q-1}v \quad \text{ in }\R^N\ ,u\leq v\text{ in }\R^N
\]
and $1<q\leq \frac{N}{N-2}$, then $u\equiv v$ in $\R^N$, without any assumption on the behavior at infinity of the couple $(u,v)$. Note that, taking $u=0$ in the previous statement, one has exactly the Liouville property discussed in the previous sections. Does the result extend to PDEs involving superlinear gradient terms and/or involving fully nonlinear operators? 
\end{opp}

A preliminary answer for fully nonlinear rotationally invariant operators perturbed with zero-th order terms with superlinear power growth is provided by the following result, which is based on the so-called transitivity of the viscosity inequalities.

\begin{thm}\label{lcp}
Let $1\leq q\leq\frac{\beta}{\beta-2}$ and $u,v\in C(\R^N)$ be respectively a viscosity sub- and supersolution to $\mathcal{P}^+_{\lambda,\Lambda}(D^2z)-|z|^{q-1}z=0$ in $\R^N$
satisfying $u\leq v$ in $\R^N$. Then, $u\equiv v$ in $\R^N$. 
\end{thm}

\begin{proof}
We first observe that if $u\leq v$ and they touch at some interior point $x_0$, then $u\equiv v$ in $\R^N$ by the strong comparison principle for Pucci's extremal operators, see e.g. \cite{GP}. Thus, we may assume $u<v$ in $\R^N$. By the transitivity of viscosity inequalities Lemma \ref{transvisc}, and the algebraic inequality
\[
(|v|^{q-1}v-|u|^{q-1}u)(v-u)\geq 2^{1-q}|v-u|^{q+1}\ ,
\]
it follows that $w=v-u\geq0$ solves in viscosity sense
\[
\mathcal{P}^+_{\lambda,\Lambda}(D^2w)\geq 2^{1-q}w^{q}\quad \text{ in }\R^N\ .
\]
Since $u\leq v$, we have
\[
w\geq0\ ,\quad \mathcal{P}^+_{\lambda,\Lambda}(D^2w)\geq0 \quad \text{ in }\R^N\ .
\]
Then, one can proceed as in \cite[Theorem 4.1]{CLeoni} 
 to conclude that $w\equiv0$ in $\R^N$, i.e. $u\equiv v$ in $\R^N$.
\end{proof}

\begin{opp}\label{opdist}
When the Liouville property fails, it is well-known that for problems driven by the Laplacian or the $p$-Laplacian it is possible to establish the sharp distance at infinity from the non-constant supersolution to the constant one, see e.g. \cite[Theorem 1.2]{KurtaProc}, \cite{Kurta}. Does this quantitative property extend to the fully nonlinear setting?
\end{opp}
We conclude with an open problem for truncated Laplacians, as studied in e.g. \cite{BGL}.
\begin{opp}
Do the results of the previous sections extend to the problem $-\mathcal{P}^+_k(D^2u)\geq H_i(u,Du)$, $u\geq0$ in $\R^N$, where $\mathcal{P}^+_k(X)=\sum_{i=N-k+1}^N e_i(X)$, for $2<k<N$?
\end{opp}

\section{The Liouville property for solutions}\label{sec:sol}

In this final section, we collect some remarks on the Liouville property for {\it solutions} to elliptic PDEs involving super-linear, first-order perturbations. Our starting point is the well-known result stating that any classical solution to 
\[
-\Delta u+|Du|^\gamma=0\text{ in }\R^N\ ,\gamma>1
\]
must be constant, see \cite[p.67]{SP} or \cite[Corollary IV]{Lions85}. Remarkably, no one-side bounds on $u$ are required. Such a strong result, which (partially) extends to quasi-linear problems, has important consequences in the study of the regularity and the qualitative behavior of solutions. In this direction, the literature is extensive, and we refer to \cite{VeronPre,Chang,VeronBook} for further references.

In what follows, we discuss some open questions (to our knowledge) regarding generalizations of this result to equations involving more general operators,  with particular emphasis to subelliptic and nonlocal problems. We stress that all the aforementioned results have been obtained through (local) gradient estimates, generally accomplished via refinements of the Bernstein method (see also the more recent works \cite{FPS2,VeronPre}). However, such gradient bounds seem to be difficult to be reproduced both in the subelliptic and the nonlocal frameworks due to the structure of the diffusion operator. 

\subsection{The subelliptic case}
We consider here the Liouville property for solutions to problems structured over H\"ormander vector fields, with {\it quadratic} gradient terms (that is, the so-called natural gradient growth). 

Consider a family $\X=\{X_1,...,X_m\}$, $m\leq N$, of linearly independent smooth vector fields on $\R^N$, $N\geq3$, having the following properties:

\begin{itemize}
\item[(i)] $X_i$'s are $\delta_\lambda$ homogeneous of degree one with respect to a family of non-isotropic dilations defined as
\[
\delta_\lambda:\R^N\to\R^N\ ,\delta_\lambda(x)=(\lambda^{\sigma_1}x_1,...,\lambda^{\sigma_N}x_N)
\]
where $1= \sigma_1\leq...\leq\sigma_N$ are positive integers, i.e. $X_i(\varphi(\delta_\lambda(x)))=\lambda (X_i\varphi)(\delta_\lambda(x))$ for every $\varphi\in C^\infty(\R^N)$, $x\in\R^N$, $\lambda>0$;
\item[(ii)] The system $\X$ satisfies the H\"ormander condition, see \cite{BLU} and the references therein.
\end{itemize}

We denote by $D_\X u=(X_1u,...,X_mu)\in\R^m$, $(D^2_\X u)^*=\frac{X_iX_ju+X_jX_iu}{2}\in\mathrm{Sym}_m$ respectively the horizontal gradient and Hessian of the unknown $u$, and $\Delta_\X u=\mathrm{Tr}((D^2_\X u)^*)=\sum_{i=1}^mX_i^2u$. We recall that, for instance, assumptions (i)-(ii) are satisfied by the vector fields generating stratified Lie groups and Grushin-type geometries in $\R^N$, $N\geq3$. We premise the following
\begin{lemma}[Degenerate Hopf-Cole transformation]\label{hc} Let $u$ be a classical solution to $-\Delta_\X u+b(x)\cdot D_\X u+|D_\X u|^2=f(x)$. Then $v=e^{-u}$ solves
\[
-\Delta_\X v+b(x)\cdot D_\X v+f(x)v=0\ .
\]
\end{lemma}

\begin{proof}
Using standard calculus rules in the subelliptic setting, cf \cite{BLU}, we write
\[
D_\X v=-vD_\X u
\]
\[
D^2_\X v=vD_\X u\otimes D_\X u-vD^2_\X u\implies \Delta_\X v=\mathrm{Tr}(D^2_\X v)=v(|D_\X u|^2-\Delta_\X u)\ .
\]
Then
\begin{multline*}
-\Delta_\X v+b(x)\cdot D_\X v+f(x)v=-v(|D_\X u|^2-\Delta_\X u)-vb(x)\cdot D_\X u+f(x)v\\
=v[\Delta_\X u-|D_\X u|^2-b(x)\cdot D_\X u+f(x)]=0\ .
\end{multline*}
\end{proof}

As a simple consequence one can deduce the following result, that also gives an alternative proof to \cite[p.67]{SP}, \cite[Corollary IV]{Lions85} in the quadratic case $\gamma=2$ for Euclidean vector fields.

\begin{thm}\label{thmviahopf}
Assume (i)-(ii). Let $u$ be a classical solution of 
\begin{equation}\label{hjquad2}
-\Delta_\X u+|D_\X u|^2=0\text{ in }\R^N\ .
\end{equation}
Then, $u$ is constant.
\end{thm}

\begin{proof}
We use the Hopf-Cole transform in Lemma \ref{hc} to show that if $u$ solves in classical sense \eqref{hjquad2}, then $v=e^{-u}>0$ is a positive solution to
\[
-\Delta_\X v=0\text{ in }\R^N\ .
\]
We then apply the Liouville property in \cite[Corollary 8.3]{KL}, \cite[Proposition 5.5]{BiagiLanconelli} (which are obtained as a consequence of Harnack inequalities) to conclude that $v$ is constant since it is one-side bounded. Then, also $u$ is constant.
\end{proof}

\begin{rem}
A similar property holds for solutions to nonlinear problems driven by the stationary Kolmogorov operator 
\[
v\cdot D_x u-\Delta_vu+|D_vu|^2=0\text{ in }\R^{2N}\ ,(x,v)\in \R^N\times\R^N\ .
\]
Indeed, $w=e^{-u}>0$ solves $v\cdot D_x w-\Delta_v w=0$ in $\R^{2N}$ and  $w$ is constant by the Liouville property obtained in \cite{KL} via the Harnack inequality, since $w$ is one-side bounded.
\end{rem}

While the proof in the special quadratic case is rather easy (basically it requires a linear result only), the non-quadratic case appears to be challenging. Indeed, classical results \cite[p.67]{SP} or \cite[Corollary IV]{Lions85} are consequence of gradient bounds (obtained via the so-called Bernstein method). It is not clear at this stage how to prove analogous bounds in the subelliptic framework.

\smallskip

We now give some comments on possible generalizations to the fully nonlinear framework. The problem seems to be open even in the Euclidean setting.

\begin{opp}
Let $\gamma > 1$ and $u$ be a viscosity solution to 
\[
\mathcal{P}^\pm_{\lambda,\Lambda}(D^2u)+|Du|^\gamma=0\text{ in }\R^N\  .
\]
Can we conclude that $u$ is a constant (without assuming any bound on $u$)?
\end{opp}
We emphasize that even in the quadratic case $\gamma=2$, the result cannot be deduced through the Hopf-Cole transformation, since such change of variable leads to an inequality (as in Theorem \ref{rep}), rather than to an equality. Still, we mention that some results in these direction recently appeared in \cite{FQ}.

More generally, having in mind the works \cite{SP,Lions85,FPS2,VeronDuke}, one could formulate the
\begin{opp}
Let $F((D^2_\X u)^*)=\mathcal{P}^\pm_{\lambda,\Lambda}((D^2_\X u)^*)$ or the $p$-sub-Laplacian $-\Delta_{p,\X}u=-\mathrm{div}_\X(|D_\X u|^{p-2}D_\X u)$, $\mathrm{div}_\X(\Phi(x))=\sum_{i=1}^mX_i\Phi$, $\Phi:\R^N\to\R^m$. Does the one (or two) side Liouville property for solutions to
\[
F((D^2_\X u)^*)\pm u^q|D_\X u|^\gamma=0\text{ in }\R^N\ ,q>0\ ,\gamma>1\ ,
\]
hold? Does the property hold without bounds on the solution as in \cite{SP,Lions85}?
\end{opp}

\subsection{Nonlocal problems}
The derivation of Liouville properties for solutions to equations driven by nonlocal operators, like the fractional Laplacian $(-\Delta)^s$, $s\in(0,1)$, perturbed by gradient terms with superlinear nature, seems to be also a nontrivial issue.
We mention the following

\begin{opp}
Does the Liouville property hold for solutions to
\[
\mathcal{L}_s u\pm|Du|^\gamma=0\text{ in }\R^N\ ,
\]
where $\mathcal{L}_s=-\Delta +(-\Delta)^s$ or $\mathcal{L}_s=(-\Delta)^s$, $s\in(0,1)$, $\gamma>0$?
\end{opp}

As in the subelliptic setting, the derivation of gradient estimates via Bernstein-type arguments seems to be not feasible (or, at least, not straightforward), and a different procedure might be required. In this direction, we mention some recent results obtained in \cite{Jakobsen} for linear equations, which combine PDE and group theory techniques. 


\small

\bibliographystyle{amsplain}



%
%
%
%
%
%
%


\providecommand{\bysame}{\leavevmode\hbox to3em{\hrulefill}\thinspace}
\providecommand{\MR}{\relax\ifhmode\unskip\space\fi MR }
\providecommand{\MRhref}[2]{%
  \href{http://www.ams.org/mathscinet-getitem?mr=#1}{#2}
}
\providecommand{\href}[2]{#2}

\end{document}